\documentclass[a4paper,10pt,reqno,twoside]{amsart}
\usepackage{amsmath,amsthm,amssymb,enumerate,mathtools,stmaryrd,diagbox}
\usepackage{pgfplots}
\usepackage{cancel,soul}
\usepackage[latin1]{inputenc}
\usepackage{bbm}
\usepackage{esint}
\usepackage{multirow}
\usepackage{euscript,mathrsfs}
\usepackage{xcolor}
\usepackage[left=3.25cm,right=3.25cm,top=3.0cm,bottom=2.5cm]{geometry}
\usepackage[colorlinks=true, linktocpage=true, linkcolor=red!70!black, citecolor=green!50!black, urlcolor=black]{hyperref}
\usepackage{breakurl}
\usepackage{colortbl}

\mathtoolsset{showonlyrefs}

\allowdisplaybreaks
\usepackage{tikz}
\usepackage{scalerel}

\usetikzlibrary{decorations.pathmorphing}

\tikzset{snake it/.style={decorate, decoration=snake}}

\usepackage[colorinlistoftodos,prependcaption,textsize=tiny]{todonotes}

\usepackage{enumitem}
\setenumerate{label={\rm (\alph{*})}}

\numberwithin{equation}{section}

\newcommand{\bbone}{\text{\usefont{U}{bbold}{m}{n}1}}
\MakeRobust{\bbone}
\newcommand{\R}{\mathbb R}

\newcommand{\D}{\mathrm{D}}

\newcommand{\lebe}{\operatorname{L}}
\newcommand{\sobo}{\operatorname{W}}
\newcommand{\hold}{\operatorname{C}}

\newcommand{\ball}{\mathrm{B}}
\newcommand{\dashint}{\fint}

\newcommand{\locc}{\mathrm{loc}}

\newcommand{\dif}{\!\mathrm{d}}
\newcommand{\mres}{\mathbin{\vrule height 1.6ex depth 0pt width
0.13ex\vrule height 0.13ex depth 0pt width 1.3ex}}

\newcommand{\bv}{\mathrm{BV}}

\theoremstyle{plain}
\newtheorem{theorem}{Theorem}[section]
\newtheorem{lemma}[theorem]{Lemma}
\newtheorem{proposition}[theorem]{Proposition}
\newtheorem{corollary}[theorem]{Corollary}

\theoremstyle{remark}
\newtheorem{remark}[theorem]{Remark}

\begin{document}
\numberwithin{equation}{section}

\colorlet{RED}{red}
\title[Sharp ellipticity thresholds]{The sharp ellipticity threshold \\ for $\sobo^{1,1}$-regularity}
\author[F. Gmeineder]{{Franz Gmeineder}}
\address[Franz Gmeineder]{Department of Mathematics, Technical University of Darmstadt, Schlossgartenstrasse 7, 64289 Darmstadt, Germany}
\email{gmeineder@mathematik.tu-darmstadt.de}
\email{franz.gmeineder@uni-konstanz.de \textrm{\emph{(expiring address, valid until September 30, 2026)}}}

\date\today

\maketitle
\begin{abstract} 
We establish that the ellipticity of the minimal surface integrand is the sharp threshold for $\sobo^{1,1}$-regularity of solutions to the Dirichlet problem on $\bv$. Here, we particularly deal with the autonomous case of convex linear growth functionals. 
\end{abstract}
\tableofcontents

\section{Introduction}\label{sec:introduction}
Functionals with linear growth are a central topic in the Calculus of Variations, largely so because they provide a far-reaching generalization and more systematic understanding of the classical (non-parametric) \emph{area} or \emph{minimal surface} integral 
\begin{align}\label{eq:minsurf}
\mathscr{A}[u;\Omega]\coloneqq \int_{\Omega}\sqrt{1+|\nabla u|^{2}}\,\,\dif x,\qquad u\in\sobo^{1,1}(\Omega),
\end{align}
where $\Omega\subset\R^{n}$ is open and bounded with Lipschitz boundary. More precisely, we say that a convex integrand $F\colon\R^{N\times n}\to\R$ is \emph{of linear growth} if there exist $c_{1},c_{2},c_{3}>0$ such that 
\begin{align}\label{eq:lingrowth}
c_{1}|z|-c_{2}\leq F(z)\leq c_{3}(1+|z|)\qquad\text{for all}\;z\in\R^{N\times n}.   
\end{align}
Once a Dirichlet datum $u_{0}\in\sobo^{1,1}(\Omega;\R^{N})$ is prescribed, the associated variational principle
\begin{align}\label{eq:Dirichlet}
\text{to minimize}\;\;\;\mathscr{F}[u;\Omega] \coloneqq \int_{\Omega}F(\nabla u)\,\,\dif x\;\;\text{over}\;\sobo_{u_{0}}^{1,1}(\Omega;\R^{N})\coloneqq u_{0}+\sobo_{0}^{1,1}(\Omega;\R^{N})
\end{align}
is easily seen to produce bounded minimizing sequences $(u_{j})$ in $\sobo^{1,1}(\Omega;\R^{N})$. However, by potential concentration effects, such sequences might fail to have weakly precompact gradients in $\lebe^{1}(\Omega;\R^{N\times n})$. This, in turn, is seen easiest for the modulus function $F=|\cdot|$ in one dimension. Such concentration effects motivate the passage to the space $\bv(\Omega;\R^{N})$ of \emph{functions of bounded variation}. By definition, a map $u\in\lebe^{1}(\Omega;\R^{N})$ belongs to $\bv(\Omega;\R^{N})$ if and only if its distributional gradient is an $\R^{N\times n}$-valued finite Radon measure, in formulas $\D u\in\mathrm{RM}_{\mathrm{fin}}(\Omega;\R^{N\times n})$. Since $\mathscr{F}[-;\Omega]$ is a priori defined for $\sobo^{1,1}$-maps, the direct method consequently requires an extension (or relaxation) of $\mathscr{F}[-;\Omega]$ to $\bv(\Omega;\R^{N})$ by lower semicontinuity. Following the work of Goffman \& Serrin \cite{GOFSER64} and Reshetnyak \cite{RESHETNYAK68} (see also Bildhauer \cite{BILDHAUER03} and Giusti \cite{GIUSTI77}), the requisite lower semicontinuous hull has the integral representation 
\begin{align}\label{eq:integral}
\begin{split}
\overline{\mathscr{F}}_{u_{0}}^{*}[u;\Omega]  & = \int_{\Omega}F(\nabla u)\,\,\dif x  + \int_{\Omega}F^{\infty}\Big(\,\frac{\dif\D^{s}u}{\dif|\D^{s}u|}\Big)\,\,\dif|\D^{s}u| \\ & + \int_{\partial\Omega}F^{\infty}(\mathrm{tr}_{\partial\Omega}(u_{0}-u)\otimes\nu_{\partial\Omega})\,\,\dif\mathscr{H}^{n-1},\qquad u\in\bv(\Omega;\R^{N}). 
\end{split}
\end{align}
In \eqref{eq:integral}, we employed the usual Lebesgue-Radon-Nikod\'{y}m decomposition 
\begin{align}\label{eq:LRN}
\D u = \D^{a} u+\D^{s}u = \nabla u\,\mathscr{L}^{n}+\frac{\dif\D^{s}u}{\dif|\D^{s}u|}|\D^{s}u|
\end{align}
of $\D u\in\mathrm{RM}_{\mathrm{fin}}(\Omega;\R^{N\times n})$ into its absolutely continuous part $\D^{a}u$ and singular part $\D^{s} u$ with respect to $\mathscr{L}^{n}$. Moreover, by convexity and \eqref{eq:lingrowth}, the \emph{recession function} $F^{\infty}(z)\coloneqq \lim_{t\searrow 0}tF(\frac{z}{t})$ is well-defined. It is well known that $\overline{\mathscr{F}}_{u_{0}}^{*}[-;\Omega]$ is in fact the right extension, meaning that 
\begin{align}\label{eq:nogap}
\inf_{\sobo_{u_{0}}^{1,1}(\Omega;\R^{N})}\mathscr{F}[-;\Omega] = \min_{\bv(\Omega;\R^{N})}\overline{\mathscr{F}}_{u_{0}}^{*}[-;\Omega], 
\end{align}
and it is implicit in the no-gap-identity \eqref{eq:nogap} that minimizers of the relaxed functional exist indeed. In what follows, any such minimizer shall be called a $\bv$-\emph{minimizer} of $\mathscr{F}[-;\Omega]$ from \eqref{eq:Dirichlet}. For more detail on $\bv$-functions and relaxations, we refer the reader to \S \ref{sec:BV} below. 

Returning to the area integral \eqref{eq:minsurf}, it has long been known that $\bv$-minimizers are of class $\sobo^{1,1}$, being equivalent to the absence of the singular part $\D^{s}u$ in \eqref{eq:LRN}. The latter is a regularity feature specific to linear growth functionals, and does not emerge in the study of superlinear growth problems. For the area integral, this can even be strengthened to real analyticity (see, e.g., Giusti \cite{GIUSTI77}), and particularly makes the appearance of the recession integral over $\Omega$ in \eqref{eq:integral} a posteriori superfluous. On the other hand, the positively $1$-homogeneous integrand $F=|\cdot|$ does not necessarily produce $\sobo^{1,1}$-regular $\bv$-minimizers. 

It is thus particularly important to exhibit sharp criteria on the integrands $F$ which lead to the genuine $\sobo^{1,1}$-regularity of $\bv$-minimizers, a question that we answer in this paper. To put this task into its natural context, we firstly collect the requisite underlying terminology and revisit  the presently available results. 

\subsection{Non-uniform and $\mu$-ellipticity}\label{sec:nonuni} Inspired by Bernstein's genre \cite{Bernstein1912} and the conditions considered by Ladyzhenskaya and Ural'tseva \cite{LadyzhenskayaUralceva1968,LadyzhenskayaUralceva1970}, the natural notion of ellipticity for integrands with linear growth is that of $\mu$-\emph{ellipticity}. More precisely, given $\mu>1$, a $\hold^{2}$-integrand $F\colon\R^{N\times n}\to\R$ is said to $\mu$-elliptic if there exist constants $0<\lambda\leq \Lambda<\infty$ such that 
\begin{align}\label{eq:muell}
\lambda(1+|z|)^{-\mu}|\xi|^{2} \leq \langle \nabla ^{2}F(z)\xi,\xi\rangle \leq \Lambda (1+|z|)^{-1}|\xi|^{2}\qquad\text{for all}\;z,\xi\in\R^{N\times n}. 
\end{align}
This notion is due to Bildhauer and Fuchs \cite{BILDHAUERFUCHS2002MU,BILDHAUER03} in the linear growth situation (see also Fuchs and Mingione \cite{FuchsMingione2000}) and gives a precise bound on the corresponding ellipticity ratio 
\begin{align*}
\mathcal{R}_{F}(R)\coloneqq \sup_{|z|\leq R} \frac{\text{highest eigenvalue of $\nabla^{2}F(z)$}}{\text{lowest eigenvalue of $\nabla^{2}F(z)$}} \approx (1+R)^{\mu-1}.
\end{align*}
Since $\mathcal{R}_{F}(R)\to\infty$ as $R\to\infty$, integrands with \eqref{eq:muell} are also called \emph{non-uniformly elliptic}, e.g., in the terminology of De Filippis and Mingione \cite{DEFILIPPIS23,DEFMIN25,DEFMIN25b}. As an important special case, we note that the area integrand $\sqrt{1+|z|^{2}}$ is $3$-elliptic and not $\mu$-elliptic for any $1<\mu<3$. The requirement $\mu>1$ is necessary, since $1$-elliptic integrands do not exist; see the discussion in Beck and Schmidt \cite{BECSCH15}. From \eqref{eq:muell} and $\mu>1$, it is clear that integrands of linear growth share a striking resemblance with those of $(p,q)$-growth on the level of second derivatives. Here, for $1<p\leq q<\infty$, one typically has bounds 
\begin{align}\label{eq:pqsecond}
\lambda(1+|z|)^{p-2}|\xi|^{2}\leq \langle \nabla^{2}F(z)\xi,\xi\rangle \leq \Lambda (1+|z|)^{q-2}|\xi|^{2}\qquad\text{for all}\;z,\xi\in\R^{N\times n}.
\end{align}
Initiated in the works of Marcellini \cite{MARCELLINI86,MARCELLINI89,MARCELLINI91}, the systematic study of the non-uniformly elliptic behaviour has attracted large attention over the past decades; see \cite{BECMIN20,BILFUC01,BILFUC03,BOUSQUETLLEDOS25,ESPLEOMIN99,ESPLEOMIN04,CarozzaKristensenPasarelli,BELSCH20,BELSCH24,DEFKOCKRI24,MARPAP06,MINGIONE06Dark,SCHAEFFNER21} for a selection of contributions. It is by now well known that improved Sobolev regularity of minimizers requires a balance between $p$ and $q$. In general, for $1<p\leq q<\infty$ such a balance is dimensional and reads as 
\begin{align}\label{eq:pqbound1}
\frac{q}{p}<1+\frac{2}{n-1}, 
\end{align}
which can be improved to $q\leq p+2$ if minimizers are a priori known to be locally bounded, see \cite{CarozzaKristensenPasarelli,DEFKOCKRI24,SCHAEFFNER21}. The range displayed in \eqref{eq:pqbound1} extends the range 
\begin{align}\label{eq:pqbound2}
\frac{q}{p}<1+\frac{2}{n}
\end{align}
as appearing, e.g., in \cite{ESPLEOMIN99}. For the following, we note that all of such exponent restrictions come with a quantitative gain of gradient integrability, that is, $\sobo_{\locc}^{1,r}$-regularity for some  $r=r(p,q,n)>p$ as a consequence of degenerate weighted second order estimates. This also applies to functionals of almost linear growth, for which we refer the reader to De Filippis et al. \cite{DeFilippisDeFilippisPiccinini2026} and the references therein.

To motivate the link between \eqref{eq:muell} in the linear growth case and \eqref{eq:pqsecond} for $1<p\leq q<\infty$, one identifies $\mu\triangleq q-2$ and $p-2\triangleq -1$. Rewriting \eqref{eq:pqbound2} or $q\leq p+2$ as a bound on the difference of $(q-2)$ and $(p-2)$ (see also \cite[\S 1.3]{BECEITGME}), one arrives at 
\begin{align}\label{eq:muconditions}
\mu<1+\frac{2}{n}\;\;\;(\text{unconditional case})\;\;\;\text{and}\;\;\;\mu\leq 3\;\;\;(\text{$\lebe_{\locc}^{\infty}$-constrained case}).
\end{align}
Following the foundational work of Bildhauer \cite{BILDHAUER2002MUBOUNDED,BILDHAUERFUCHS2002MU,BILDHAUER03} and later by Beck and Schmidt \cite{BECSCH13}, the conditions \eqref{eq:muconditions} are by now known to yield $\sobo^{1,1}$-regularity of $\bv$-minimizers. Moreover, as discussed in \cite[Remark 5.14]{BeckGmeinederSchaeffner2026}, the first condition on $\mu$ from \eqref{eq:muconditions} is presently not known to be improvable to $\frac{2}{n-1}$ as in \eqref{eq:pqbound1}. This identifies $\mu=3$ as a crucial borderline case, in turn being precisely the non-uniform ellipticity of the area integrand. Here, one still obtains the regularity assertion $u\in\sobo^{1,1}(\Omega;\R^{N})$ and $\nabla u\in\lebe\log^{2}\lebe_{\locc}(\Omega;\R^{N\times n})$ for locally bounded $\bv$-minimizers $u$. This leads to the central conjecture that, when going beyond $\mu=3$, higher gradient integrability might persist on a more fine-tuned scale, e.g., $\nabla u\in\lebe\log^{\alpha(\mu)}\lebe_{\locc}$ with $\alpha(\mu)\to 0$ as $\mu\to\infty$. While such findings would still yield a quantitative logarithmic gain in gradient integrability, a substantially weaker yet more approachable objective is the purely qualitative $\sobo^{1,1}$-regularity of $\bv$-minimizers for $\mu>3$. Such results have been obtained for Neumann problems with radially symmetric integrands and $\sobo^{2,\infty}$-regular forcing in \cite{BECBULGME20}, but do \emph{not} apply to the Dirichlet problem on $\bv$ as considered here. Deferring a more precise discussion of this matter to \S \ref{sec:approaches} below, all of the available techniques break down beyond the threshold $\mu=3$ for \eqref{eq:Dirichlet}. 

As the main result of the present paper, even a purely qualitative $\sobo^{1,1}$-regularity is impossible for $\mu>3$. More precisely, we have: 
\begin{theorem}[Main Theorem]\label{thm:main}
Let $\mu>3$ and let $n= 2$, $N= 1$. Then there exists a $\mu$-elliptic integrand $F\in\hold^{2}(\R^{2})$, an open and bounded Lipschitz domain $\Omega\subset\R^{2}$ and a Dirichlet datum $u_{0}\in\sobo^{1,1}(\Omega)$ such that there exists a $\bv$-minimizer of $\mathscr{F}[-;\Omega]$ from \eqref{eq:Dirichlet} with 
\begin{align}
u\in(\bv(\Omega)\cap\lebe^{\infty}(\Omega))\setminus\sobo^{1,1}(\Omega).
\end{align}
We express the membership $u\in\bv(\Omega)\setminus\sobo^{1,1}(\Omega)$ by saying that $u$ is a \emph{singular $\bv$-minimizer} (of $\mathscr{F}[-;\Omega]$ from \eqref{eq:Dirichlet}). 
\end{theorem}
Theorem \ref{thm:main} identifies $\mu=3$, that is, the ellipticity of the area integrand, as the decisive threshold for $\sobo^{1,1}$-regularity. As discussed in detail below in \S\ref{sec:outline}, the key point in Theorem \ref{thm:main} is that the integrands are in fact autonomous, meaning that they do not share an additional $x$-dependence. In fact, in such scenarios, $\bv$-minimizers $u\in\bv(\Omega;\R^{N})\setminus\sobo^{1,1}(\Omega;\R^{N})$ are easier to obtain. First examples are due to Giaquinta, Modica and Sou\v{c}ek \cite{GIAMODSOU79} and Bildhauer \cite[Chapter 4.4]{BILDHAUER03}. In the latter contribution, an integrand $F\in\hold^{2}(\overline{\Omega}\times\R^{N\times n})$ is constructed which is $\mu$-elliptic uniformly in the first variable, but allows for a singular $\bv$-minimizer. Here, the key idea is that, despite the benign $x$-dependence, this additional flexibility does \emph{not} make it energetically unfavourable for a $\bv$-minimizer to exhibit jumps; a   discussion thereof can also be found in the recent work of De Filippis \cite[\S 1.1]{DeFilippis2026}. A similar idea appears in the work of Fu\ss angel, Priyasad and Stephan \cite{FUSPRISTE25}, in turn being based on Kaiser \cite{Kaiser1975}. In the situation considered here, however, the respective integrands are autonomous, and so analogous energetic considerations are difficult to implement. Let us note that Theorem \ref{thm:main} can be lifted to higher dimensions $n\geq 2$ and $N\geq 2$ by relatively routine means; here, we focus on the base case displayed in Theorem \ref{thm:main} and shall give the details in the higher dimensional case elsewhere. 

From the perspective of Sobolev regularity, Theorem \ref{thm:main} asserts that $\bv$-minimizers are essentially not more regular than generic $\bv$-maps. Our particular construction, however, hinges on a jump across a hyperplane. Thus, it does not rule out the absence of Cantor parts for $\mu>3$. The author has tried to obtain such a $\mathrm{SBV}$-regularity result even for $\mu>3$, but has not been successful in doing so. 

Yet, in light of Theorem \ref{thm:main}, partial $\hold^{1,\alpha}$-regularity of $\bv$-minimizers is close to optimal. Here, we recall that $u\in\bv(\Omega;\R^{N})$ is \emph{partially $\hold^{1,\alpha}$-regular} if there exists an open set $\Omega_{u}\subset\Omega$ with $\mathscr{L}^{n}(\Omega\setminus\Omega_{u})=0$ such that $u$ is of class $\hold^{1,\alpha}$ for any $0<\alpha<1$ in a neighbourhood of any $x_{0}\in\Omega_{u}$. In this situation, $\Sigma_{u}\coloneqq\Omega\setminus\Omega_{u}$ is called the \emph{singular set}.

Such results have been established first in the convex linear growth case by Anzellotti and Giaquinta \cite{ANZGIA88}, and have been lifted to the quasiconvex linear growth case by Kristensen and the author  \cite{GMEKRI19b}; see also \cite{GMEINEDER16,GMEINEDER20,GMEINEDER21,GMEKRI24,Li2022,Schmidt2014}. It is important to note that, when working in the realm of $\mu$-elliptic integrands, these partial $\hold^{1,\alpha}$-regularity assertions are independent of the specific choice of $\mu$. Hence, Theorem \ref{thm:main} and the discussion afterwards imply the following lower bound for the Hausdorff dimensions of the underlying singular sets $\Sigma_{u}$: 
\begin{corollary}\label{cor:singular}
In the situation of Theorem \ref{thm:main}, \emph{every $\bv$-minimizer is partially $\hold^{1,\alpha}$-regular}, and the Hausdorff dimension bound from below  
\begin{align}\label{eq:Hausdorfflower}
\dim_{\mathscr{H}}(\Sigma_{u})\geq 1 
\end{align}
holds \emph{for one $\bv$-minimizer $u\in\bv(\Omega)$}. 
\end{corollary}
This result can be lifted to the higher dimensional case too, in which case \eqref{eq:Hausdorfflower} reads $\dim_{\mathscr{H}}(\Sigma_{u})\geq n-1$. 
Let us note that the preceding corollary only makes an assertion on one singular  $\bv$-minimizer. Indeed, despite the strict convexity of $\mu$-elliptic integrands, $\bv$-minimizers generally fail to be unique. This has been observed by Santi  \cite{SANTI72} and Finn \cite{Finn1965} for minimal surface equations, and has been studied in detail in the general vectorial case by Beck and Schmidt \cite{BECSCH13}.
The chief reason for this phenomenon is the appearance of the positively $1$-homogeneous recession function $F^{\infty}$ in \eqref{eq:integral}. Since $F^{\infty}$ acts on the singular parts exclusively, this leads to the relaxed functional $\overline{\mathscr{F}}_{u_{0}}^{*}[-;\Omega]$ being non-strictly convex. This in itself is a major obstruction to obtain universal regularity estimates (that is, for all $\bv$-minimizers) if $\mu\leq 3$, see \cite{BECSCH13,GMEKRI19,GMEINEDER20,BeckGmeinederSchaeffner2026}. Now, \emph{if} $\bv$-minimizers are known to be $\sobo^{1,1}$-regular, the boundary penalization integral in \eqref{eq:integral} is the sole source of potential non-uniqueness. More precisely, if $\Omega\subset\R^{n}$ is moreover connected, then the set of all $\bv$-minimizers is contained in the affine space $v+\R a$ for some $a\in\R^{N}$ and one $\bv$-minimizer $v\in\sobo^{1,1}(\Omega;\R^{N})$. This result is essentially sharp if no curvature conditions on the boundary are imposed. Yet, it is precisely Theorem \ref{thm:main} which rules out the crucial $\sobo^{1,1}$-regularity for $\mu>3$ as required for the aforementioned conclusions, and so the lower bound \eqref{eq:Hausdorfflower} cannot be extended to all $\bv$-minimizers by trivial means. 
\begin{remark}\label{rem:DF}
As addressed in detail in \S\ref{sec:outline} below, we  note that the integrand $F$ underlying Theorem \ref{thm:main} is not radially symmetric. For such integrands, De Filippis \cite{DeFilippis2026} has recently obtained sharp criteria on $\mu$ for the underlying variational integrals to produce Non-Lipschitzian but H\"{o}lder regular $\sobo^{1,1}$-minimizers, thereby answering a question dating back to Ladyzhenskaya and Ural'tseva; see also Beck et al. \cite{BECBULMAR18} and the references therein  for positive results in the linear growth context. We point out that the construction in the proof of Theorem \ref{thm:main} does not work for radially symmetric integrands, and it is conceivable that radial symmetry leads to $\sobo^{1,1}$-regularity of $\bv$-minimizers indeed. 
\end{remark}
\subsection{Organization of the paper} In \S \ref{sec:prelims}, we fix notation and gather background results on $\bv$-functions, relaxations and convex analysis. In \S \ref{sec:outline}, we revisit previous approaches to the construction of irregular minimizers and give an outline of our particular strategy. Following the latter, we construct in \S \ref{sec:integrands} the integrands underlying Theorem \ref{thm:main}. Finally, in \S\ref{sec:singularity}, we establish Theorem \ref{thm:main}, from where Corollary \ref{cor:singular} follows immediately.  
\subsection*{Acknowledgments} The author  gratefully acknowledges funding through the Hector Foundation II (Project Number FP 626/21), and is moreover grateful to the Oxford Centre for Nonlinear Partial Differential Equations (OxPDE) for the funding of a visit in March 2026, during which this project was commenced. Lastly, he is also thankful for the invitation to the Oberwolfach workshop 'Calculus of Variations' in August/September 2026, during which the project was finished.

\section{Preliminaries}\label{sec:prelims}
\subsection{General notation}\label{sec:notation}
Throughout, $\Omega\subset\R^{n}$ denotes an open and bounded set. We use $\langle\cdot,\cdot\rangle$ to denote the Euclidean inner product on $\R^{N}$ or $\R^{N\times n}$. To alleviate notation, we sometimes write '$\cdot$' instead, but no confusions will arise from this. We write $|\cdot|$ for the underlying norm; for $A\in\R^{N\times n}$, we write $\|A\|$ for its operator norm with respect to the Euclidean norm. For $x_{0}\in\R^{n}$ and $r>0$, we write 
\begin{align*}
\ball_{r}(x_{0})\coloneqq \{x\in\R^{n}\colon\;|x-x_{0}|<r\}, 
\end{align*}
and abbreviate $\mathbb{S}^{n-1}\coloneqq \partial\ball_{1}(0)$ for the unit sphere in $\R^{n}$.  Moreover, we write $\mathbbm{1}_{n\times n}$ for the unit matrix in $\R^{n\times n}$; for $a\in\R^{N}$ and $b\in\R^{n}$, we denote by $a\otimes b\coloneqq ab^{\top}$ the associated tensor product. For a symmetric matrix $A\in\R_{\mathrm{sym}}^{n\times n}$, we  use  $\lambda_{A}^{-}$ for its smallest, and $\lambda_{A}^{+}$ for its largest eigenvalue. In order to distinguish from  boundary trace operators, we write $\mathrm{Tr}(A)$ for the trace of a quadratic matrix and  $\mathrm{tr}_{\partial\Omega}$ for the corresponding boundary trace operators, see also \S \ref{sec:BV} below. 

As usual, we write $\mathscr{L}^{n}$ and $\mathscr{H}^{n-1}$ for the $n$-dimensional Lebesgue-- and the $(n-1)$-dimensional Hausdorff measure.  For a finite dimensional inner product space $V$, we denote by $\mathrm{RM}_{\mathrm{fin}}(\Omega;V)$ the $V$-valued, finite Radon measures; if $V=\R$, we simply write $\mathrm{RM}_{\mathrm{fin}}(\Omega)$. Lastly, for $\mu\in\mathrm{RM}_{\mathrm{fin}}(\Omega;V)$ and a Borel set $A\subset\Omega$, we write $\mu\mres A$ for its restriction to $A$. 
\subsection{Functions of bounded variation and relaxations}\label{sec:BV}
We now record some background facts on functions of bounded variation and relaxations, all of which is covered in more detail, e.g., in \cite{AMBFUSPAL00,EVANSGARIEPY}.
For an open set $\Omega\subset\R^{n}$, the space $\bv(\Omega;\R^{N})$ of functions of bounded variation is defined as the collection of all $u\in\lebe^{1}(\Omega;\R^{N})$ such that 
\begin{align}\label{eq:TV}
|\D u|(\Omega)\coloneqq \sup \Big\{\int_{\Omega}\langle u,\mathrm{div}(\varphi)\rangle\,\,\dif x \colon\;\varphi\in\hold_{c}^{\infty}(\Omega;\R^{N\times n}),\;\;\|\varphi\|_{\lebe^{\infty}(\Omega)}\leq 1 \Big\}<\infty, 
\end{align}
and we refer to the preceding quantity as the \emph{total variation} of $u$. In particular, the distributional gradient satisfies $\D u\in\mathrm{RM}_{\mathrm{fin}}(\Omega;\R^{N\times n})$, and \eqref{eq:TV} is its total variation. Below, we shall also refer to $\D u$ as the \emph{measure gradient} of $u$. The norm on $\bv(\Omega;\R^{N})$ is given by $\|u\|_{\bv(\Omega)}\coloneqq \|u\|_{\lebe^{1}(\Omega)}+|\D u|(\Omega)$, and we simply write $\bv(\Omega)=\bv(\Omega;\R)$ if $N=1$.  

Let $u\in\bv(\Omega;\R^{N})$. As alluded to in the introduction, the Lebesgue--Radon--Nikod\'{y}m decomposition of $\D u$ into its absolutely continuous and singular parts with respect to $\mathscr{L}^{n}$ reads  
\begin{align}\label{eq:LNR}
\D u =\D^{a}u + \D^{s}u = \nabla u\mathscr{L}^{n}\mres\Omega + \frac{\dif\D^{s}u}{\dif|\D^{s}u|}|\D^{s}u|,
\end{align}
where $\nabla u$ is the so-called \emph{approximate gradient}. Based on \eqref{eq:LNR}, we may further split 
\begin{align*}
\D^{s}u = \D^{j}u + \D^{c} u, 
\end{align*}
where $\D^{j}u\coloneqq \D^{s}u\mres \mathrm{J}_{u}$ is the \emph{jump part} of $\D u$ and $\D^{c}u\coloneqq \D^{s}u\mres\mathrm{L}_{u}$ is the \emph{Cantor part} of $\D u$. Here, the set $\mathrm{J}_{u}$ is defined as the collection of all $x_{0}\in\Omega$ for which there exist $\nu\in\mathbb{S}^{n-1}$ and $a^{+},a^{-}\in\R^{N}$ with $a^{+}\neq a^{-}$ such that 
\begin{align*}
\lim_{r\searrow 0}\dashint_{\ball_{r,\nu}^{+}(x_{0})}|u-a^{+}|\,\,\dif x = \lim_{r\searrow 0}\dashint_{\ball_{r,\nu}^{-}(x_{0})}|u-a^{-}|\,\,\dif x = 0,
\end{align*}
where $\ball_{r,\nu}^{\pm}(x_{0})=\{x\in\ball_{r}(x_{0})\colon\;\langle x-x_{0},\nu\rangle \gtrless 0\}$. Here, the dash indicates a mean value integral with respect to $\mathscr{L}^{n}$. Lastly, $\mathrm{L}_{u}$ is the set of all Lebesgue points of $u$.

In the following, let $\Omega$ be open and bounded with Lipschitz boundary. Since the canonical norm topology is too strong for most purposes, one introduces a substitute as follows: Given $u,u_{1},...\in\bv(\Omega;\R^{N})$, we say that $(u_{j})$ converges to $u$ in the \emph{weak*-sense on $\bv(\Omega;\R^{N})$} if and only if $u_{j}\to u$ strongly in $\lebe^{1}(\Omega;\R^{N})$ and $\D u_{j}\stackrel{*}{\rightharpoonup}\D u$ in the weak*-sense on $\mathrm{RM}_{\mathrm{fin}}(\Omega;\R^{N\times n})\cong\hold_{0}(\Omega;\R^{N\times n})'$. In particular, this yields compactness in the following sense: If $(u_{j})\subset\bv(\Omega;\R^{N})$ is such that $\sup_{j\in\mathbb{N}}\|u_{j}\|_{\bv(\Omega)}<\infty$, then there exists $u\in\bv(\Omega;\R^{N})$ and a subsequence $(u_{j_{i}})\subset (u_{j})$ such that $u_{j_{i}}\to u$ in the weak*-sense on $\bv(\Omega;\R^{N})$. 

For the following, we note that there exists a surjective linear (boundary) trace operator $\mathrm{tr}_{\partial\Omega}\colon\bv(\Omega;\R^{N})\to\lebe^{1}(\partial\Omega;\R^{N})$ such that $\mathrm{tr}_{\partial\Omega}(u)=u|_{\partial\Omega}$ holds $\mathscr{H}^{n-1}$-a.e. on $\partial\Omega$ for all  $u\in\bv(\Omega;\R^{N})\cap\hold(\overline{\Omega};\R^{N})$. This trace operator is continuous with respect to $\|\cdot\|_{\bv(\Omega)}$, but not so with respect to weak*-convergence on $\bv(\Omega;\R^{N})$. As briefly discussed below, this is the main reason for the boundary penalization integrals appearing in \eqref{eq:integral}. 

If $N=1$, the classical integration-by-parts formula on $\bv(\Omega)$ asserts that 
\begin{align}\label{eq:IBP0}
\int_{\Omega}\left\langle\varphi,\,\,\dif\D u\right\rangle + \int_{\Omega}\mathrm{div}(\varphi) u\,\,\dif x = \int_{\partial\Omega}\mathrm{tr}_{\partial\Omega}(u)\langle\varphi,\nu_{\partial\Omega}\rangle\,\,\dif\mathscr{H}^{n-1}
\end{align}
holds for all $\varphi\in\hold^{1}(\overline{\Omega};\R^{n})$ and all $u\in\bv(\Omega)$, where $\nu_{\partial\Omega}\colon\partial\Omega\to\mathbb{S}^{n-1}$ is the outer unit normal to $\partial\Omega$. As usual, the first term on the left-hand side of \eqref{eq:IBP0} is defined as 
\begin{align*}
\int_{\Omega}\langle\varphi,\,\dif\D u \rangle= \int_{\Omega}\left\langle\varphi,\frac{\dif\D u}{\dif|\D u|}\right\rangle\,\,\dif|\D u|. 
\end{align*}
Below, we shall require a version of \eqref{eq:IBP0} for merely continuous and  distributionally solenoidal test maps. This is certainly clear to the experts, and is recorded here for the sake of completeness:
\begin{lemma}\label{lem:IBP}
Let $\Omega,U\subset\R^{n}$ be open and bounded such that $\Omega\Subset U$. Moreover, let $\sigma\in\hold(U;\R^{n})$ be such that $\mathrm{div}(\sigma)=0$ in $\mathscr{D}'(U)$, and let $u\in\bv(\Omega)$. Then we have 
\begin{align}\label{eq:dividentity}
\int_{\Omega}\langle\sigma,\,\dif\D u\rangle = \int_{\partial\Omega}\mathrm{tr}_{\partial\Omega}(u)\langle \sigma,\nu_{\partial\Omega}\rangle\,\,\dif \mathscr{H}^{n-1}. 
\end{align}
\end{lemma}
\begin{proof}
By construction, $\sigma$ is distributionally solenoidal and continuous in a neighbourhood of $\Omega$. For the $\varepsilon$-rescaled variant $\rho_{\varepsilon}=\frac{1}{\varepsilon^{n}}\rho(\frac{\cdot}{\varepsilon})$ of a standard mollifier $\rho$, $\rho_{\varepsilon}*\sigma$ is well-defined on $\Omega$ for all sufficiently small $\varepsilon>0$. In consequence, we have by the Gau\ss --Green formula on $\bv$: 
\begin{align*}
\int_{\Omega}\langle\rho_{\varepsilon}*\sigma,\,\dif\D u\rangle & = \int_{\partial\Omega}\mathrm{tr}_{\partial\Omega}(u)\langle\rho_{\varepsilon}*\sigma,\nu_{\partial\Omega}\rangle\,\,\dif\mathscr{H}^{n-1} - \int_{\Omega}\underbrace{\mathrm{div}(\rho_{\varepsilon}*\sigma)}_{=0} u\,\,\,\dif x \\ & =  \int_{\partial\Omega}\mathrm{tr}_{\partial\Omega}(u)\langle \rho_{\varepsilon}*\sigma,\nu_{\partial\Omega}\rangle\,\,\dif\mathscr{H}^{n-1}.
\end{align*}
Sending $\varepsilon\searrow 0$, we have $\rho_{\varepsilon}*\sigma \to \sigma$ uniformly on $\overline{\Omega}$. Because of $\D u\in\mathrm{RM}_{\mathrm{fin}}(\Omega;\R^{n})$ and $\mathrm{tr}_{\partial\Omega}(u)\in\lebe^{1}(\partial\Omega)$, the validity of \eqref{eq:dividentity} follows from Lebesgue's theorem on dominated convergence. The proof is complete. 
\end{proof}
To keep the paper self-contained, we conclude this subsection with some more detailed comments on the relaxed functionals \eqref{eq:integral}. For this, let $u_{0}\in\sobo^{1,1}(\Omega;\R^{N})$ and let $F\in\hold(\R^{N\times n})$ be convex with \eqref{eq:lingrowth}. In this situation, the extended functional \eqref{eq:integral} coincides with the Lebesgue--Serrin--Marcellini relaxation (see, e.g., \cite{MARCELLINI86}) 
\begin{align*}
\overline{\mathscr{F}}_{u_{0}}^{*}[u;\Omega]= \inf\left\{ \liminf_{j\to\infty}\int_{\Omega}F(\nabla u_{j})\,\,\dif x \colon\; \begin{array}{c} (u_{j})\subset \sobo_{u_{0}}^{1,1}(\Omega;\R^{N}), \\ u_{j}\stackrel{*}{\rightharpoonup} u\;\text{in}\;\bv(\Omega;\R^{N})\end{array}\right\},\qquad u\in\bv(\Omega;\R^{N}).
\end{align*}
This is the lower semicontinuous hull with respect to weak*-convergence on $\bv(\Omega;\R^{N})$, and thus identifies \eqref{eq:integral} as the canonical extension of $\mathscr{F}[-;\Omega]$ from $\sobo_{u_{0}}^{1,1}(\Omega;\R^{N})$ to $\bv(\Omega;\R^{N})$ by lower semicontinuity. Hence, the direct method implies the existence of $\bv$-minimizers, and the no-gap-result \eqref{eq:nogap} in turn is a consequence of a refined trace-preserving smooth approximation result \cite[Chapter B.1]{BILDHAUER03}. We recall that the underlying \emph{recession function} $F^{\infty}\colon\R^{N\times n}\to\R$ is defined by 
 \begin{align}\label{eq:recessionfunction}
F^{\infty}(z)\coloneqq \lim_{t\searrow 0}tF\Big(\frac{z}{t}\Big),\qquad z\in\R^{N\times n},
 \end{align}
 and exists and is finite due to convexity and the linear growth of $F$. For our future purposes and as mentioned in \S\ref{sec:nonuni}, we explicitly record that 
 \begin{align}\label{eq:1homorecess}
F^{\infty}\;\;\;\text{is positively $1$-homogeneous}. 
 \end{align}
 In \eqref{eq:integral}, the boundary penalization term governed by the recession function is due to the fact that the trace operator $\mathrm{tr}_{\partial\Omega}\colon\bv(\Omega;\R^{N})\to\lebe^{1}(\partial\Omega;\R^{N})$ is not continuous with respect to weak*-convergence on $\bv(\Omega;\R^{N})$. More precisely, as shown in \cite{Finn1965,SANTI72}, see also \cite{BECSCH13}, $\bv$-minimizers indeed do not need to attain prescribed Dirichlet data; instead, the relaxed functional forces the deviation from the prescribed data to be minimal in the interplay with the bulk energy. For more detail, the reader is referred to \cite{AMBFUSPAL00,BILDHAUER03}.
\subsection{Convex analysis}\label{sec:convex}
In this subsection, we gather some background facts from convex analysis as required in the main part. These can be traced back to \cite{HiriartUrrutyLemarechal2001,Rockafellar1970}. First, a convex function $f\colon\R^{n}\to\R\cup\{+\infty\}$ is said to be \emph{proper} if $f\not\equiv +\infty$. In this situation, $\mathrm{dom}(f)$ is the set of all points $x\in\R^{n}$ such that $f(x)\neq +\infty$. For a proper convex function $f$, its \emph{Fenchel conjugate} is given by 
\begin{align}\label{eq:fenchel}
f^{*}(y)\coloneqq \sup_{x\in\R^{n}}\langle x,y\rangle - f(x),\qquad y\in\R^{n}. 
\end{align}
Moreover, for $x_{0}\in\mathrm{dom}(f)$, we say that $A\in\R^{n}$ is a \emph{subgradient} of $f$ at $x_{0}$ if 
\begin{align*}
f(x_{0})+\langle A,x-x_{0}\rangle\leq f(x)\qquad\text{for all}\;x\in\R^{n}. 
\end{align*}
The \emph{subdifferential} $\partial f(x_{0})$ of $f$ at $x_{0}$ is the set of all subgradients of $f$ at $x_{0}$. Then, if $f$ is differentiable at $x_{0}$, one has $\partial f(x_{0})=\{\nabla f(x_{0})\}$. 

As a consequence of \eqref{eq:fenchel}, one immediately gets the Fenchel--Young inequality 
\begin{align}\label{eq:FenchelYoung}
\langle x, y\rangle \leq f(x)+f^{*}(y)\qquad \text{for all}\;x,y\in\R^{n}. 
\end{align}
For a proper convex and lower semicontinuous function $f\colon\R^{n}\to\R\cup\{+\infty\}$, the cases of equality in \eqref{eq:FenchelYoung} are characterized by
\begin{align}\label{eq:FenchelLegendre}
\langle x,y\rangle = f(x) + f^{*}(y) \Longleftrightarrow y\in \partial f(x) \Longleftrightarrow x \in\partial f^{*}(y),
\end{align}
see \cite[Corollary E.1.4.4]{HiriartUrrutyLemarechal2001}. Next, we record a classical result on biconjugation: 
\begin{lemma}[Fenchel--Moreau, {\cite[Corollary 1.3.6]{HiriartUrrutyLemarechal2001}}]\label{lem:fenchelmoreau}
Let $f\colon\R^{n}\to \R\cup\{+\infty\}$ be a proper, convex, and lower semicontinuous function. Then there holds $f=f^{**}(\eqqcolon(f^{*})^{*})$.
\end{lemma}
Finally, in order to deal with recession functions in the main part, we require the following characterization via support functions. Here, we directly specify to functions of linear growth. 
\begin{lemma}[{\cite[Theorem 8.5, Theorem 13.3]{Rockafellar1970}}]\label{lem:fenchel}
Let $f\colon \R^{n}\to \R \cup \{+\infty\}$ be a convex and lower semicontinuous function of linear growth, meaning that \eqref{eq:lingrowth} holds with $F$ being replaced by $f$. Then the recession function $f^{\infty}$ from \eqref{eq:recessionfunction} can be expressed as 
\begin{align*}
f^{\infty}(z) \coloneqq  \lim_{t\searrow 0}tf\Big(\frac{z}{t}\Big)= \sup_{y \in {\mathrm{dom}(f^{*})}} \langle y, z\rangle \qquad\text{for all}\;z\in\R^{n}. 
\end{align*}
\end{lemma}
\section{Previous approaches and outline of the proof of Theorem \ref{thm:main}}\label{sec:outline} 
In this section, we briefly revisit the available strategies to arrive at Sobolev regularity for $1<\mu\leq 3$, and discuss the available examples that lead to singular $\bv$-minimizers. Based on the latter, we give a detailed outline of our approach to be executed in \S\ref{sec:integrands} and \S\ref{sec:singularity}. 
\subsection{Approaches to $\sobo^{1,1}$-regularity and available thresholds}\label{sec:approaches}
We start by explaining where the assumption of $\mu\leq 3$ enters the existing proofs of $\sobo^{1,1}$-regularity of locally bounded $\bv$-minimizers; this summarizes the strategies developed in \cite{BILDHAUER2002MUBOUNDED,BECSCH13,GMEINEDER20,BECEITGME,SCHMIDTHabil}. Given a locally bounded $\bv$-minimizer $u$, one firstly constructs a sequence $(u_{j})\subset \sobo_{\locc}^{2,2}(\Omega;\R^{N})\cap\bv(\Omega;\R^{N})\cap\lebe_{\locc}^{\infty}(\Omega;\R^{N})$ such that 
\begin{align}\label{eq:strat}
\mathscr{F}[u_{j};\Omega] \to \overline{\mathscr{F}}_{u_{0}}^{*}[u;\Omega]\;\;\;\text{and}\;\;\;\sup_{j\in\mathbb{N}}\int_{\Omega'}\frac{|\nabla^{2}u_{j}|^{2}}{(1+|\nabla u_{j}|)^{\mu}}\,\,\dif x <\infty
\end{align}
for every relatively compact set $\Omega'\Subset\Omega$. To circumvent the potential non-uniqueness of $\bv$-minimizers, this is achieved by combining classical vanishing viscosity methods with Ekeland's variational principle \cite{EKELAND74}, and improved versions to keep local $\lebe^{\infty}$-bounds have been given in \cite{BECEITGME,SCHMIDTHabil}. While being well suited for universal regularity estimates, this approach comes to the effect that the $u_{j}$'s only satisfy an Euler-Lagrange \emph{inequality}
\begin{align}\label{eq:perturbed}
\left\vert \int_{\Omega}\langle \nabla F(\nabla u_{j}),\nabla\varphi\rangle\,\,\dif x  + (\text{viscosity stabilization terms)} \right\vert\leq \frac{1}{j}\|\varphi\|_{\sobo^{-1,1}(\Omega)}
\end{align}
for all $\varphi\in\sobo_{0}^{1,1}(\Omega;\R^{N})$. It is now that one utilizes \eqref{eq:strat} by testing \eqref{eq:perturbed} with special maps that give access to the $\lebe\log^{2}\lebe_{\locc}$-estimates. Dating back to Bildhauer \cite{BILDHAUER2002MUBOUNDED,BILDHAUER03}, the state-of-the-art choice still is $\varphi=\rho^{2}\log^{2}(1+|\nabla u_{j}|^{2})u_{j}$. Inserting this map into \eqref{eq:perturbed} yields three leading terms, depending on which factor is covered by the product rule. By the uniform local boundedness of the $u_{j}$'s, the left-hand side of the overall inequality produces the quantity $\rho^{2}\log^{2}(1+|\nabla u_{j}|^{2})\nabla u_{j}$, while the right-hand side of the overall inequality then contains the critical term 
\begin{align}\label{eq:critico}
\rho^{2}\left\vert\nabla \left( \log^{2}(1 + |\nabla u_{j}|^{2}) \right)\right\vert |u_{j}|  & \lesssim \rho^{2}\frac{\log(1 + |\nabla u_{j}|^2)}{(1 + |\nabla u_{j}|^2 )^{\frac{1}{2}}} |\nabla^{2}u_{j}| \coloneqq \mathrm{I}_{j}, 
\end{align}
where we recall that $u_{j}$ is uniformly locally bounded. In order to absorb the critical singularity into the overall left-hand side, one is bound to apply Young's inequality as follows: 
\begin{align*}
\mathrm{I}_{j} \leq \varepsilon \rho^{2}\log^{2}(1+|\nabla u_{j}|^{2})(1+|\nabla u_{j}|^{2})^{\frac{1}{2}} + \rho^{2} \frac{|\nabla^{2}u_{j}|^{2}}{(1+|\nabla u_{j}|^{2})^{\frac{3}{2}}}. 
\end{align*}
By \eqref{eq:strat}, the ultimate term is only controllable if $\mu\leq 3$. We note that higher powers of the logarithmic factors in $\varphi$ come with a similar issue; see \cite[Remark 4.11]{BECEITGME} in the related symmetric gradient case.

The only $\mu$-independent $\sobo^{1,1}$-regularity result we are aware of is due to Beck, Buli\v{c}ek and the author \cite{BECBULGME20}, where the radially symmetric \emph{Neumann problem}
\begin{align}\label{eq:neumann}
\text{to minimize}\;\;\;\int_{\Omega}f(|\nabla u|)-g\cdot \nabla u\,\,\dif x \;\;\text{over}\;\;\;\sobo^{1,1}(\Omega;\R^{N})\;\text{with}\;(u)_{\Omega}=0
\end{align}
with an open and bounded, simply connected Lipschitz domain is considered. Here, the inhomogeneity satisfies $g\in\sobo^{2,\infty}(\Omega;\R^{N})$. The outcome is that, if $F=f(|\cdot|)\in\hold^{2}(\R^{N\times n})$ satisfies $F''>0$ with a natural upper growth bound, then solutions of \eqref{eq:neumann} in fact belong to $\sobo^{1,1}(\Omega;\R^{N})$. The proof strategy, however, does not hinge on quantitative estimates. Instead, letting $(u_{j})$ be a suitable vanishing viscosity sequence, the biting limit of $(\nabla u_{j})$ in the sense of \cite{BallMurat1989} is shown to be curl-free. Hence, because $\Omega$ is simply connected, it is a gradient itself, and is subsequently shown to lead to a minimizer of \eqref{eq:neumann}. In view of Theorem \ref{thm:main}, we firstly note that the result from \cite{BECBULGME20} is only known to hold for radially symmetric integrands, which are not covered by Theorem \ref{thm:main}; see also Remark \ref{rem:DF}. Secondly, despite the radial symmetry of $F=f(|\cdot|)$, the proof strategy from \cite{BECBULGME20} does \emph{not} apply to the Dirichlet problem on $\bv$. In particular, in these scenarios, Theorem \ref{thm:main} does not make any assertion on the potential irregularity of $\bv$-minimizers. 

Returning to the Dirichlet problem on $\bv$, Bildhauer \cite[Chapter 4.4]{BILDHAUER03} constructed an integrand $\hold^{2}(\overline{\Omega}\times\R^{2})$ which is of linear growth and $\mu$-elliptic uniformly in $x\in\overline{\Omega}$, yet possesses a singular $\bv$-minimizer $u$. From a construction viewpoint, this approach is a sophisticated generalization of the strategy employed in the classical work \cite[Example 3.2]{GIAMODSOU79} by Giaquinta, Modica \& Sou\v{c}ek via radial symmetry. For $p>2$, $N=n=1$ and $\Omega=(-1,1)$, the latter  hinges on the functional 
\begin{align}\label{eq:GMS}
\mathscr{F}[u;\Omega] \coloneqq \int_{-1}^{1}(1+(1+|x|^{2})|u'(x)|^{p})^{\frac{1}{p}},\qquad u\in\sobo^{1,1}((-1,1)), 
\end{align}
which is minimized subject to 
\begin{align*}
u(-1)=-m,\;\;\;u(1)=m,\qquad\text{where}\;m>0\;\text{is such that}\;\int_{-1}^{1}((1+x^{2})^{\frac{p}{p-1}}-1)^{-\frac{1}{p}}\,\,\dif x < m.
\end{align*}
Even though for fixed $p\in(1,\infty)\setminus\{2\}$, the integrand in \eqref{eq:GMS} is not $\mu$-elliptic (in fact, this is due to its singular or degenerate behaviour for the value $z=0$ of the gradient variable, see \cite{BECSCH15}), it can be made into a $\mu=p+1$-elliptic integrand. Now, as established in \cite{GIAMODSOU79}, it is indeed the minimum $x=0$ of the weight $(1+|x|^{2})$ appearing in \eqref{eq:GMS} which makes its energetically less unfavourable for a minimizer to exhibit a jump at $x=0$. It is important to note that this phenomenon occurs despite the $\hold^{2}$-regularity in $x$ and even though the corresponding weight is bounded away from zero. Yet, the underlying $x$-dependence is essential for this sort of argument; this methodological point is already visible in Mania's related classical work \cite{Mania1934}.

Making a radially symmetric ansatz $F=f(|\cdot|)$ in the autonomous case (that is, without $x$- or $u$-dependence) to create singular $\bv$-minimizers comes with several obstructions. One of the main issues is that, when considering the stress $\sigma$ (which corresponds to $F'(\nabla u)=f'(|\nabla u|)\frac{\nabla u}{|\nabla u|}$ in the absolutely continuous setting, see \cite{BeckSchmidt2015} for the general $\bv$-situation), $\sigma$ is necessarily parallel to $\nabla u$. In our approach below, we crucially use that already the approximate gradient of the singular $\bv$-minimizer satisfies an Euler-Lagrange equation by itself. When aiming to come up with a $\bv$-minimizer with the analogous property in the radially symmetric setting, it is this rigidity reflected by $\sigma$ being parallel to $\nabla u$ which seems to rule out separation of variable-approaches as employed in \S\ref{sec:singularity} below. This in itself, however, does not exclude the possibility of creating non-trivial Cantor parts on suitable fractals. In the context of double phase functionals and their generalizations, limiting examples and irregular minimizers have been constructed in \cite{BalciDieningSurnachev2020,BalciDieningSurnachev2025,FonsecaMalyMingione2004}, but crucially use the explicit $x$-dependence of the integrands. In particular, it is unclear how similar strategies can be made to work in the framework of Theorem \ref{thm:main}.

\subsection{Strategy in the proof of Theorem \ref{thm:main}}
Let $n=2$ and $N=1$. In view of Theorem \ref{thm:main} and by the above discussion, it is natural to seek an integrand $F\colon\R^{2}\to\R$ which is anisotropic. In the related context of $(p,q)$-type functionals, explicit examples are given in \cite{Giaquinta1987,Marcellini1987,MARCELLINI89}, but these pertain to superlinear growth scenarios and cannot be modified by easy means to apply to the situation at our disposal. In view of Theorem \ref{thm:main}, the simplest way to come up with an anisotropic, $\mu$-elliptic integrand is to require eigenvalue scales in $e_{1}$-- and $e_{2}$--directions by 
\begin{align*}
\langle \nabla^{2}F(z)e_{1},e_{1}\rangle \lesssim z_{2}^{-1}\lesssim |z|^{-1}\;\;\;\text{and}\;\;\;\langle\nabla^{2}F(z)e_{2},e_{2}\rangle \gtrsim |z|^{-\mu}
\end{align*}
at least for large $z_{1},z_{2}>0$ with $z_{1}\lesssim z_{2}$. In this case, the second requirement reduces to $\langle\nabla^{2}F(z)e_{2},e_{2}\rangle \gtrsim z_{2}^{-\mu}$. A suitable ansatz thus is to consider 
\begin{align*}
\mathfrak{f}(z_{1},z_{2})= Az_{2}^{2-\mu} + b\frac{z_{1}^{2}}{z_{2}}\qquad\text{for large $z_{1},z_{2}>0$ such that $z_{1}\lesssim z_{2}$}
\end{align*}
for $A,b>0$. This particular choice is of linear growth in the far upper halfplane $\mathbb{H}$ subject to $z_{1}\lesssim z_{2}$. Finally aiming to construct a minimizer which has a jump in the $e_{2}$-direction, the recession function of the integrand must be non-trivial at $e_{2}$, for otherwise the jump is ignored in the relaxed functional. Since $\mathfrak{f}^{\infty}(e_{2})=0$, we modify $\mathfrak{f}$ to 
\begin{align}\label{eq:2F0}
F_{0}(z_{1},z_{2})=Lz_{2} + Az_{2}^{2-\mu}+b\frac{z_{1}^{2}}{z_{2}}
\end{align}
for some $L>0$. On suitable cones $\mathcal{C}_{R,\delta}\subset\mathbb{H}\coloneqq \{z_{2}>0\}$ opened in $e_{2}$-direction with $\mathrm{dist}(\mathcal{C}_{R,\delta},\partial\mathbb{H})>0$ (see \eqref{eq:cones} and Figure \ref{fig:cone}), this function is indeed $\mu$-elliptic and of linear growth; this is established in Lemma \ref{lem:muelliptic1}. Next note that $F_{0}$ blows up as $z_{2}\searrow 0$, whereby $F_{0}$ cannot be extended to an integrand of linear growth on $\R^{2}$. Our plan is therefore to consider $F_{0}$ on cones $\mathcal{C}_{R,\delta}$ first and then to construct a convex $\hold^{2}$-integrand $F\in\hold^{2}(\R^{n})$ such that $F=F_{0}$ on $\mathcal{C}_{R,\delta}$. In parallel, $F$ must remain $\mu$-elliptic. That this is possible is indeed one of the main results of the present paper, see Theorem \ref{thm:muellipticextension}. 

Before we address the construction of this extension $F$, let us note that the identity $F=F_{0}$ on $\mathcal{C}_{R,\delta}$ is primarily motivated by the eventual construction of a singular $\bv$-minimizer: If $F=F_{0}$ on cones $\mathcal{C}_{R,\delta}$, the latter opening in $e_{2}$-direction, $\frac{1}{t}e_{2}\in \mathcal{C}_{R,\delta}$ holds for all sufficiently small $t>0$. Hence the recession function satisfies $F^{\infty}(e_{2})=F_{0}^{\infty}(e_{2})=L$ and thus is capable to see non-trivial jumps in $e_{2}$-direction. 

In the following, we provide some heuristics. For \emph{any} linear growth integrand $F$ and $\sobo^{1,1}$-map $v$, the stress $\sigma=\nabla F(\nabla v)$ is bounded and contained in the effective domain $\mathrm{dom}(F^{*})$. Approximating $\bv$-maps $u$ by $\sobo^{1,1}$-maps $v_{j}$, $\sigma_{j}=\nabla F(\nabla v_{j}(x))$ will converge to $\partial\mathrm{dom}(F^{*})$ at singular points $x$. Hence, large values of $\nabla u$ or non-trivial densities $\frac{\dif\D^{s}u}{\dif|\D^{s}u|}$ are contained in small neighbourhoods of $\partial\mathrm{dom}(F^{*})$; see \S \ref{sec:convex} for this terminology. Since we aim to construct $F$, this dual viewpoint suggests to firstly extend $F_{0}$ to $\R^{2}$ via $F_{0}=+\infty$ on $\R^{2}\setminus\overline{\mathbb{H}}$ and to then examine its Fenchel conjugate.  The effective domain $\mathrm{dom}(F_{0}^{*})$, however, is unbounded from below, a fact that is linked to $F_{0}=+\infty$ on $\R^{2}\setminus\overline{\mathbb{H}}$. To arrive at a linear growth integrand, we thus modify the effective domain and make it into a convex and \emph{bounded} set $K$. As mentioned above, this must be accomplished in a way such that the dual variable corresponding to $\frac{\dif\D^{s}u}{\dif|\D^{s}u|}=e_{2}$ is unaffected by this modification. As will be discussed below, this is precisely the dual point $(0,L)$ (that is, the vertex in the parabola displayed in Figure \ref{fig:mod}), in a neighbourhood of which $K$ and $\mathrm{dom}(F_{0}^{*})$ have to coincide.

Calling the emerging integrand $G_{0}$, $F\coloneqq G_{0}^{*}$ has linear growth by construction; recall that $K$ is bounded. Moreover, the cones $\mathcal{C}_{R,\delta}$ in the primal picture correspond to small neighbourhoods of $(0,L)$ in $K$, which is why $F=F_{0}$ holds on such cones (see Proposition \ref{prop:coincidence}). During this modification process, we moreover have to ensure that $F$ is in fact $\mu$-elliptic. This is the content of Proposition \ref{prop:Fmuell}.

All of the preceding steps solely require $\mu>2$. It is thus the construction of the particular $\bv$-minimizer $u$ which requires  $\mu>3$. The construction underlying the absolutely continuous part $\D^{a}  u=\nabla v\mathscr{L}^{2}$ for $v\in\sobo^{1,1}(\Omega)$ hinges on a product ansatz. This crucially exploits the structure of the integrand $F_{0}$, so that the distributional solenoidality of the associated stress $\nabla F(\nabla v)$ can essentially be reduced to an ordinary differential equation (see Lemma \ref{lem:wconstruct}). This, in turn, determines the underlying set $\Omega$.  Since $F$ is defined abstractly as the Fenchel conjugate, gaining access to $F=F_{0}$ requires the inclusion $\nabla v\in\mathcal{C}_{R,\delta}$ $\mathscr{L}^{2}$-a.e. in $\Omega$ (see Lemma \ref{lem:contained1}). As a key point, this construction can only yield $\nabla v\in\lebe^{1}(\Omega;\R^{2})$ if $\mu>3$. Adding a jump to $v$ finally yields the requisite map, and it is then shown that this map is a $\bv$-minimizer with respect to its own boundary values; this shall settle Theorem \ref{thm:main}.

\section{Construction of the $\mu$-elliptic integrand $F$}\label{sec:integrands}
\begin{figure}
\begin{tikzpicture}
\draw[->] (0,-0.2) -- (0,2.75); 
\draw[->] (-4,0) -- (4,0);
\draw[-,white,fill=blue!20!white,opacity=.4] (-4, 2.5) -- (-1.4,0.874) -- (1.4,0.874) -- (4,2.5) -- (-4,2.5);
\draw[-,dotted] (-4, 2.5) -- (0,0) -- (4,2.5);
\draw[-,dotted] (-4, 2.5) -- (0,0) -- (4,2.5);
\draw[-,blue] (-4, 2.5) -- (-1.4,0.874) -- (1.4,0.874) -- (4,2.5);
\node at (1.5,0.32) {$x_{2}=\frac{1}{\delta}|x_{1}|$};
\node[right] at (4,0) {$x_{1}$};
\node[above] at (0,2.75) {$x_{2}$};
\draw[dashed] (-4,0.874) -- (4,0.874); 
\node[right] at (4,0.874) {$R$};
\node[blue!50!white] at (1.25,1.75) {\large $\mathcal{C}_{R,\delta}$};
\end{tikzpicture}
\caption{The cones from \eqref{eq:cones}, which shall eventually contain the gradients of a singular $\bv$-minimizer.}
\label{fig:cone}
\end{figure}
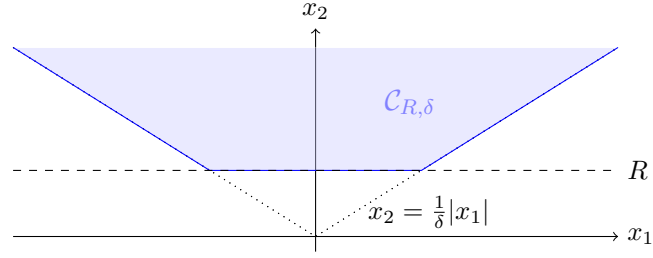
In this section, we give the detailed construction of the integrand $F$ that underlies Theorem \ref{thm:main}. Throughout, we suppose that $\mu>2$.  Moreover, let $L,A,b>0$. In what follows, we denote the upper halfspace in $\R^{2}$ by $\mathbb{H}\coloneqq \R\times (0,\infty)$. As explained in \S \ref{sec:outline}, see \eqref{eq:2F0}, we begin by defining $F_{0}\colon\mathbb{H}\to\R$ via  
\begin{align}\label{eq:main}
F_{0}(z) \coloneqq Lz_{2}+Az_{2}^{2-\mu}+b\frac{z_{1}^{2}}{z_{2}},\qquad z=(z_{1},z_{2})\in\mathbb{H},  
\end{align}
which we initially consider on cones 
\begin{align}\label{eq:cones}
\mathcal{C}_{R,\delta}\coloneqq \{z=(z_{1},z_{2})\in\R^{2}\colon\;z_{2}>R\;\;\text{and}\;\;\;|z_{1}|<\delta z_{2}\},
\end{align}
where $R,\delta>0$, see Figure \ref{fig:cone}. The chief reason for considering $F_{0}$ on $\mathcal{C}_{R,\delta}$ is that, on such cones, $F_{0}$ from \eqref{eq:main} in fact has linear growth, with underlying constants depending on $R$ and $\delta$. Moreover, we clearly have $F_{0}\in\hold^{2}(\mathbb{H})$. For future reference, we compute for $z\in\mathbb{H}$
\begin{align}\label{eq:firstvariation}
\nabla F_{0}(z)=\left(\begin{array}{c} \frac{2bz_{1}}{z_{2}} \\ L + (2-\mu)Az_{2}^{1-\mu}-\frac{b z_{1}^{2}}{z_{2}^{2}}\end{array} \right)
\end{align}
and
\begin{align}\label{eq:2ndderiv}
\begin{split}
\nabla^{2}F_{0}(z) & = \left( \begin{matrix}  \frac{2b}{z_{2}} & - \frac{2bz_{1}}{z_{2}^{2}} \\ - \frac{2bz_{1}}{z_{2}^{2}} & (2-\mu)(1-\mu)Az_{2}^{-\mu}+2\frac{bz_{1}^{2}}{z_{2}^{3}}\end{matrix}\right) \\ 
& = \frac{2b}{z_{2}}\left(\begin{matrix} 1 \\ -\frac{z_{1}}{z_{2}}\end{matrix} \right)\otimes \left(\begin{matrix} 1 \\ -\frac{z_{1}}{z_{2}}\end{matrix} \right) + A(\mu-2)(\mu-1)z_{2}^{-\mu}e_{2}\otimes e_{2}. 
\end{split}
\end{align}
We explicitly record that $F_{0}$ is $\mu$-elliptic on the cones from \eqref{eq:cones}.  
\begin{lemma}\label{lem:muelliptic1}
Let $\mu>2$, $R\geq 1$ and $\delta>0$. Then there exist constants $c=c(A,b,\mu,\delta)>0$ and $C=C(A,b,\mu,\delta)>0$, both independent of $R$, such that 
\begin{align}\label{eq:muellthefirst!}
c(1+|z|)^{-\mu}|\xi|^{2}\leq \langle \nabla^{2} F_{0}(z)\xi,\xi\rangle \leq C(1+|z|)^{-1}|\xi|^{2}
\end{align}
holds for all $z\in\mathcal{C}_{R,\delta}$ and all $\xi\in\R^{2}$.
\end{lemma}
\begin{proof}
This is a consequence of \eqref{eq:2ndderiv}. Indeed, for $z\in\mathcal{C}_{R,\delta}$ and $\xi\in\R^{2}$, one has 
\begin{align}\label{eq:refreshingbeverage}
\langle \nabla^{2}F_{0}(z)\xi,\xi\rangle & = \frac{2b}{z_{2}}\Big(\xi_{1}-\xi_{2}\frac{z_{1}}{z_{2}}\Big)^{2}+ A(\mu-2)(\mu-1)z_{2}^{-\mu}\xi_{2}^{2}. 
\end{align}
Fix an arbitrary $\varepsilon>0$ such that
\begin{align}\label{eq:wearehere}
    0<\varepsilon<\frac{1}{(1+\delta^{2})^{\frac{1}{2}}},\;\;\;\text{so that}\;\;\;\vartheta(\varepsilon,\delta)\coloneqq (1-\varepsilon^{2})^{\frac{1}{2}}-\varepsilon\delta>0.
\end{align}
If $|\xi_{2}|\geq \varepsilon |\xi|$, we immediately get from \eqref{eq:refreshingbeverage} the lower bound 
\begin{align}\label{eq:orchestrion}
\langle \nabla^{2}F_{0}(z)\xi,\xi\rangle & \geq \varepsilon^{2}A(\mu-2)(\mu-1)(1+|z|)^{-\mu}|\xi|^{2}.
\end{align}
Now suppose that $|\xi_{2}|<\varepsilon|\xi|$. Then we have 
\begin{align}\label{eq:jotel}
\xi_{1}^{2} = |\xi|^{2}-\xi_{2}^{2} >(1-\varepsilon^{2})|\xi|^{2}
\end{align}
and therefore 
\begin{align*}
\left\vert\xi_{1}-\xi_{2}\frac{z_{1}}{z_{2}} \right\vert  & \geq |\xi_{1}| - |\xi_{2}|\left\vert\frac{z_{1}}{z_{2}}\right\vert \stackrel{z\in\mathcal{C}_{R,\delta}}{\geq}|\xi_{1}|-\delta|\xi_{2}|  \\ &  \!\!\!\!\!\!\!\! \!\!\!\!\!\stackrel{\eqref{eq:jotel},\,|\xi_{2}|<\varepsilon|\xi|}{\geq}\Big((1-\varepsilon^{2})^{\frac{1}{2}}-\varepsilon\delta\Big)|\xi| \stackrel{\eqref{eq:wearehere}}{\geq}\vartheta(\varepsilon,\delta)|\xi|.
\end{align*}
Since $z_{2}>0$, this implies by use of $|z|\leq (1+|z|)^{\mu}$ that 
\begin{align*} 
\langle \nabla^{2}F_{0}(z)\xi,\xi\rangle \geq \frac{2b}{|z|}\vartheta(\varepsilon,\delta)^{2}|\xi|^{2}\geq 2b\vartheta(\varepsilon,\delta)^{2}(1+|z|)^{-\mu}|\xi|^{2}, 
\end{align*}
and combining the previous estimate with \eqref{eq:orchestrion}, the lower bound in \eqref{eq:muellthefirst!} follows. For the upper bound in \eqref{eq:muellthefirst!}, we first record that 
\begin{align*}
(1+|z|)^{2} & \leq 2(1+|z|^{2}) = 2(1+z_{1}^{2}+z_{2}^{2}) \stackrel{z_{2}\geq R\geq 1,\,|z_{1}|<\delta|z_{2}|}{\leq} 2(2+\delta^{2})z_{2}^{2}
\end{align*}
and so 
\begin{align}\label{eq:vicecity}
\frac{1}{z_{2}}\leq \frac{(2(2+\delta^{2}))^{\frac{1}{2}}}{1+|z|}\;\;\;\text{and}\;\;\;\frac{1}{z_{2}^{\mu}}\leq \frac{(2(2+\delta^{2}))^{\frac{\mu}{2}}}{(1+|z|)^{\mu}}\leq \frac{(2(2+\delta^{2}))^{\frac{\mu}{2}}}{1+|z|}
\end{align}
for all $z\in\mathcal{C}_{R,\delta}$. 
Going back to \eqref{eq:refreshingbeverage} and using $|z_{1}|<\delta z_{2}$ because of $z\in\mathcal{C}_{R,\delta}$, \eqref{eq:vicecity} immediately gives the upper bound in \eqref{eq:muellthefirst!}. This completes the proof. 
\end{proof}
As established in the previous lemma, $F_{0}$ is $\mu$-elliptic on the cones $\mathcal{C}_{R,\delta}$ for $R\geq 1$ and $\delta>0$. The core of this paper is contained in the following result, asserting that $F_{0}$ can be extended from $\mathcal{C}_{R,\delta}$ to the entire $\R^{2}$ as a globally $\mu$-elliptic integrand.  
 \begin{theorem}[$\mu$-elliptic extension]\label{thm:muellipticextension}
Let $\mu>2$. Then there exists a \emph{$\mu$-elliptic integrand} $F\in\hold^{2}(\R^{2})$ of linear growth with the following property: There exists a number $\delta_{0}>0$ such that, for any $0<\delta<\delta_{0}$, there exists $R=R(\delta)>1$ such that $F=F_{0}$ holds on $\mathcal{C}_{R,\delta}$. 
 \end{theorem}
The proof of the preceding theorem requires several preparations. Following the outline given in \S \ref{sec:outline}, we let $p=(p_{1},p_{2})\in\R^{2}$ and start by computing the Fenchel conjugate 
\begin{align*}
F_{0}^{*}(p)=\sup_{\substack{z_{1}\in\R \\ z_{2}>0}} z_{1}p_{1} + z_{2}p_{2} - Lz_{2}-Az_{2}^{2-\mu}-b\frac{z_{1}^{2}}{z_{2}} \eqqcolon \sup_{\substack{z_{1}\in\R \\ z_{2}>0}}f(z_{1},z_{2}).
\end{align*}
For fixed $z_{2}>0$, we first derive the extremality relation  
\begin{align*}
\partial_{z_{1}}f(z_{1},z_{2})=p_{1}-2b\frac{z_{1}}{z_{2}} \stackrel{!}{=} 0 \Longleftrightarrow z_{1}=\frac{p_{1}z_{2}}{2b}, 
\end{align*}
and it follows from $b>0$ that this indeed leads to a global maximum of $f(\cdot,z_{2})$. Hence, 
\begin{align*}
    F_{0}^{*}(p) & = \sup_{ z_{2}>0} \frac{p_{1}^{2}z_{2}}{2b} + z_{2}p_{2} - Lz_{2} - Az_{2}^{2-\mu}-\frac{p_{1}^{2}z_{2}}{4b} \\ 
    & =\sup_{z_{2}>0}-\underbrace{\Big(L-\frac{p_{1}^{2}}{4b}-p_{2}\Big)}_{\eqqcolon d} z_{2} - Az_{2}^{2-\mu} \\ 
    & = \sup_{z_{2}>0}-d z_{2}-Az_{2}^{2-\mu} \eqqcolon \sup_{z_{2}>0}f_{d}(z_{2}). 
\end{align*}
To compute the ultimate expression, we first suppose that $d=d(p)>0$. 
 The first and second derivatives of $f_{d}$ are given by
\begin{align*}
f'_{d}(z_{2})=-d-A(2-\mu)z_{2}^{1-\mu},\;\;\;f''_{d}(z_{2})=-A(2-\mu)(1-\mu)z_{2}^{-\mu}\stackrel{\mu>2,\,z_{2}>0}{<}0, 
\end{align*}
from where it particularly follows that $f_{d}$ is strictly concave. Because of $\mu>2$ and $d,A>0$, 
\begin{align*}
\lim_{z_{2}\searrow 0}f_{d}(z_{2})=\lim_{z_{2}\to\infty}f_{d}(z_{2})=-\infty, 
\end{align*}
and so the maximum of $f_{d}$ is attained at 
\begin{align*}
z_{2}=\Big(\frac{d}{A(\mu-2)}\Big)^{\frac{1}{1-\mu}}. 
\end{align*}
Therefore, we arrive at 
\begin{align*}
\sup_{z_{2}>0}f_{d}(z_{2})=-d\Big(\frac{d}{A(\mu-2)}\Big)^{\frac{1}{1-\mu}}-A\Big(\frac{d}{A(\mu-2)}\Big)^{\frac{2-\mu}{1-\mu}} = -C(A,\mu)d^{\frac{2-\mu}{1-\mu}}, 
\end{align*}
where 
\begin{align}\label{eq:hyperpower}
C(A,\mu)\coloneqq \frac{1}{(A(\mu-2))^{\frac{1}{1-\mu}}}+\frac{A^{\frac{1}{\mu-1}}}{(\mu-2)^{\frac{2-\mu}{1-\mu}}}.
\end{align}
Hence, if $d=d(p)>0$, then 
\begin{align*}
F_{0}^{*}(p)=-C(A,\mu)\Big(L-\frac{p_{1}^{2}}{4b}-p_{2}\Big)^{\frac{\mu-2}{\mu-1}}.
\end{align*}
Now suppose that $d=d(p)=0$. Then $f_{d}(z_{2})=-Az_{2}^{2-\mu}<0$ for all $z_{2}>0$, and sending $z_{2}\to\infty$ yields $F_{0}^{*}(p)=0$ in this case because of $\mu>2$. Lastly, if $d=d(p)<0$, we have 
\begin{align*}
\lim_{z_{2}\to\infty}-d z_{2} - Az_{2}^{2-\mu} \stackrel{\mu>2}{=}+\infty. 
\end{align*}
Hence, synthesizing the three cases, we end up with 
\begin{align}\label{eq:wildestdreams}
F_{0}^{*}(p)=\begin{cases} \displaystyle
-C(A,\mu)\Big(L-\frac{p_{1}^{2}}{4b}-p_{2}\Big)^{\frac{\mu-2}{\mu-1}}&\;\text{if}\;d(p)=L-\frac{p_{1}^{2}}{4b}-p_{2}>0, \\
0&\;\text{if}\; d(p)=L-\frac{p_{1}^{2}}{4b}-p_{2}=0, \\
+\infty &\;\text{if}\;d(p)=L-\frac{p_{1}^{2}}{4b}-p_{2}<0.
\end{cases}
\end{align}
To motivate our further construction, we single out the following remark: 
\begin{remark}\label{rem:extended}
We can directly extend $F_{0}\colon\mathbb{H}\to\R$ to a proper, convex and lower semicontinuous function $\widetilde{F}_{0}\colon\R^{2}\to\R\cup\{+\infty\}$ by  
\begin{align}\label{eq:extended}
\widetilde{F}_{0}(z) \coloneqq \begin{cases}
    F_{0}(z)\;\;\;\text{as in \eqref{eq:main}}&\;\text{if}\;z_{2}>0, \\ 
    +\infty &\;\text{if}\;z_{2}\leq 0. 
\end{cases}
\end{align}
Here, the lower semicontinuity follows from the continuity of $F_{0}|_{\mathbb{H}}$ and 
\begin{align*}
\lim_{z_{2}\searrow 0}F_{0}(z_{1},z_{2})=+\infty. 
\end{align*}
Then we clearly have $\widetilde{F}_{0}^{*}=F_{0}^{*}$ with $F_{0}^{*}$ as in \eqref{eq:wildestdreams}, and so the Fenchel-Moreau theorem (see Lemma \ref{lem:fenchelmoreau}) gives us $(F_{0}^{*})^{*}=\widetilde{F}_{0}$.
\end{remark}
By the previous remark, $(F_{0}^{*})^{*}$ does not have linear growth. To achieve the latter, let $h\in\hold^{\infty}(\R)$ be a function with the following properties: 
\begin{enumerate}[label=(\roman*)] \item\label{item:(i)}$h(s)=s$ for all $s\leq 2L$ and $h(s)\leq s$ for all $s\in\R$, 
\item\label{item:(ii)} $h(S_{*})=0$ for some $S_{*}>2L$, with $h(s)>0$ if and only if $s\in (0,S_{*})$, 
\item\label{item:(iii)} $h$ is concave on $\R$, satisfies $h''<0$ (and so is strictly concave) on $(2L,\infty)$, has its unique global maximum at $s_{*}\in (2L,S_{*})$ and is strictly decreasing on $(s_{*},\infty)$, 
\item\label{item:(iv)} and, as a consequence, $h'(0)=1$, $h'(S_{*})<0$, and $h(s)=0$ if and only if $s\in \{0,S_{*}\}$.
\end{enumerate}
\begin{remark}
The function $h$ can be constructed by elementary means. Define 
\begin{align*}
\eta(t)\coloneqq \begin{cases} 0&\;\text{if}\;t\leq 2L, \\ 
\exp\Big(-\frac{1}{(t-2L)^{2}}\Big)&\;\text{if}\;t> 2L, 
\end{cases}
\end{align*}
whereby $\eta\in\hold^{\infty}(\R)$ and put, for $S_{*}>2L$, 
\begin{align*}
\widetilde{c}\coloneqq \frac{S_{*}}{\int_{2L}^{S_{*}}(S_{*}-s)\eta(s)\,\,\dif s}. 
\end{align*}
Then the function $h\colon\R\to\R$ defined by 
\begin{align}\label{eq:hdef1A}
h(t)\coloneqq t-\widetilde{c}\int_{2L}^{t}(t-s)\eta(s)\,\,\dif s,\qquad t\in\R, 
\end{align}
satisfies \ref{item:(i)}--\ref{item:(iv)}. Clearly, $h$ is of class $\hold^{\infty}(\R)$. If $t\leq 2L$, $s\in (t,2L)$ implies that $\eta(s)=0$ and so $h(t)=t$. Moreover, the second summand in \eqref{eq:hdef1A} is non-negative for any $t\in\R$, whereby $h(t)\leq t$ holds for all $t\in\R$; this is \ref{item:(i)}. On the other hand, we compute 
\begin{align*}
h'(t)=1-\widetilde{c}\int_{2L}^{t}\eta(s)\,\,\dif s \;\;\;\text{and}\;\;\;h''(t)=-\widetilde{c}\eta(t),\qquad t\in\R. 
\end{align*}
In particular, $h$ is concave on $\R$ and satisfies $h''(t)<0$ as asserted in \ref{item:(iii)}. Moreover, the very definition of $\widetilde{c}$ gives us $h(S_{*})=0$. Now, we have $h'(2L)=1$, and the continuity of $\eta$ yields 
\begin{align*}
\int_{2L}^{S_{*}}(S_{*}-s)\eta(s)\,\,\dif s < S_{*}\int_{2L}^{S_{*}}\eta(s)\,\,\dif s \Longrightarrow 1<\widetilde{c}\int_{2L}^{S_{*}}\eta(s)\,\,\dif s \Longrightarrow h'(S_{*})<0. 
\end{align*}
By continuity of $h'$, there must exist a point $s_{*}\in (2L,S_{*})$ with $h'(s_{*})=0$, and because of $h''<0$ on $(2L,S_{*})$, this point is the unique maximum of $h$ on $(2L,S_{*})$. Since $h'(2L)>0$, this maximum is indeed the global maximum. From here, all of \ref{item:(i)}--\ref{item:(iv)} follow. 
\end{remark}

For future reference, the choice of $h$ we have in mind is depicted in Figure \ref{fig:functionh}. Based on \ref{item:(i)}--\ref{item:(iv)}, we introduce the set $K\subset\R^{2}$ via 
\begin{align}\label{eq:setK}
K \coloneqq \Big\{p\in\R^{2}\colon\;\rho(p)\coloneqq h(L-p_{2})-\frac{p_{1}^{2}}{4b}>0\Big\}
\end{align}
and pause to record some elementary properties of $K$: 
\begin{figure}
\begin{tikzpicture}
    \draw[->] (-0.2,0) -- (5,0);
    \draw[->] (0,-0.2) -- (0,2.5);
    \draw[-,blue,thick] (-0.2,-0.2) -- (2,2);
    \draw[-,blue,thick] (2,2) [out=45, in =120] to (4.7,-0.2);
    \draw[-] (2,0.1) -- (2,-0.1);
    \node[below]  at (2,0) {$2L$};
    \draw[-] (4.6,0.1) -- (4.6,-0.1);
    \node[below]  at (4.5,0) {$S_{*}$};
    \node[blue] at (4,1.5) {$h$};
    \draw[-] (2.5,0.1) -- (2.5,-0.1);
    \node[below]  at (2.5,-0.1) {$s_{*}$};
\end{tikzpicture}
\caption{The function $h$ with properties \ref{item:(i)}--\ref{item:(iv)}.}
\label{fig:functionh}
\end{figure}
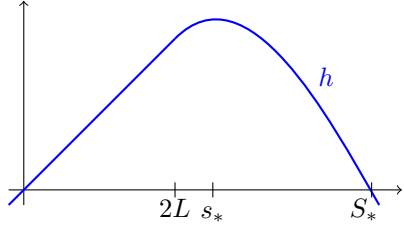
\begin{lemma}\label{lem:Kprops}
The set $K$ from \eqref{eq:setK} has the following properties:
\begin{enumerate}
    \item\label{item:K1} $K$ is open and bounded, 
     \item\label{item:K2} $K$ is convex, and 
    \item\label{item:K3} $K$ has smooth boundary.  
\end{enumerate}
\end{lemma}
\begin{proof}
On \ref{item:K1}. First $K$ is bounded. Indeed, $\rho(p)>0$ implies that $h(L-p_{2})>0$ and so $0<L-p_{2}<S_{*}$ by \ref{item:(ii)}. Hence, $L-S_{*}<p_{2}<L$. Moreover, if $\rho(p)>0$, then 
\begin{align*}
0\leq \frac{p_{1}^{2}}{4b}<h(L-p_{2}) \Longrightarrow |p_{1}|\leq 2\sqrt{b}\sup_{s\in[0,S_{*}]}\sqrt{h(s)}, 
\end{align*}
and so the boundedness follows. The openness, in turn, is clear, and so \ref{item:K1} follows. 

For \ref{item:K2} and \ref{item:K3}, we note that 
\begin{align}\label{eq:comper}
\nabla\rho(p_{1},p_{2})=\left(\begin{matrix} -\frac{p_{1}}{2b} \\ -h'(L-p_{2})\end{matrix}\right)\;\;\;\text{and}\;\;\;\nabla^{2}\rho(p_{1},p_{2})=\left(\begin{matrix} -\frac{1}{2b} & 0 \\ 0 & h''(L-p_{2})\end{matrix}\right). 
\end{align}
Since $h''\leq 0$ by \ref{item:(iii)} and $b>0$, $\nabla^{2}\rho$ is negative semidefinite and so $\rho\colon\R^{2}\to\R$ is a concave function. In particular, its superlevel sets and so $K$ is convex, which is \ref{item:K2}.  

For \ref{item:K3}, it suffices to show $\nabla\rho|_{\partial K}\neq 0$. Suppose that $p\in\partial K$, whereby $\rho(p)=0$, is such that $\nabla\rho(p)=0$. By  $\eqref{eq:comper}_{1}$, $p_{1}=0$. Because of $\rho(p)=0$ and $p_{1}=0$, this implies $h(L-p_{2})=0$, and so $p_{2}=L$ or $L-p_{2}=S_{*}$ by \ref{item:(iv)}. By \ref{item:(iv)}, it also follows that $h'(L-p_{2})\neq 0$ in both cases. This is a contradiction to $\nabla\rho(p)=0$, and so $\nabla\rho|_{\partial K}\neq 0$ follows. In consequence, $\partial K$ is a $1$-dimensional smooth manifold by the regular value theorem. This yields \ref{item:K3}, and the proof is complete. 
\end{proof}
As a by-product of the previous proof, the continuity of $\rho$ and $\nabla\rho$ implies that there exist $0<\varepsilon_{0}\leq \Theta_{0}<\infty$ and  $\varepsilon'>0$ such that we have 
\begin{align}\label{eq:boundobound}
|\rho|\leq \frac{1}{2}\;\;\text{and}\;\;\varepsilon_{0}\leq |\nabla\rho|\leq\Theta_{0}\qquad\text{on}\;K_{0}\coloneqq \{p\in K\colon\;\mathrm{dist}(p,\partial K)<\varepsilon'\}. 
\end{align}
To alleviate notation, we put  $\alpha\coloneqq \frac{\mu-2}{\mu-1}\in (0,1)$ and define, with $C(A,\mu)$ as in \eqref{eq:hyperpower}, 
\begin{align}\label{eq:swallowtears}
G_{0}(p)\coloneqq\begin{cases} \displaystyle
-C(A,\mu)\rho(p)^{\alpha}&\;\text{if}\;p\in K, \\
0&\;\text{if}\; p\in\partial K, \\
+\infty &\;\text{if}\;p\in\R^{2}\setminus \overline{K}.
\end{cases}
\end{align}
In order to form the Fenchel conjugate of $G_{0}$, we record the following intermediate lemma. 
\begin{lemma}\label{lem:convexity1}
Let $G_{0}\colon\R^{2}\to \R\cup\{+\infty\}$ be as in \eqref{eq:swallowtears}. Then, for any $p\in K$, $\nabla^{2}G_{0}(p)$ is positive definite. As a consequence, $G_{0}|_{K}$ is strictly convex and $G_{0}$ is proper, lower semicontinuous and convex. 
\end{lemma} 
\begin{proof} 
Note that $\rho|_{K}$ itself is merely concave but not strictly concave; the latter property of $\rho^{\alpha}|_{K}$ is due to the power $\alpha\in (0,1)$. More precisely, on $K$, we have $\nabla \rho(p)^{\alpha}=\alpha\rho^{\alpha-1}\nabla\rho(p)$ and so 
\begin{align}\label{eq:miami}
\nabla^{2}\rho(p)^{\alpha}= \alpha(\alpha-1)\rho(p)^{\alpha-2}\nabla\rho(p)\otimes\nabla\rho(p) + \alpha\rho(p)^{\alpha-1}\nabla^{2}\rho(p), 
\end{align}
whereby we obtain for any $v=(v_{1},v_{2})\in\R^{2}\setminus\{0\}$ that 
\begin{align*}
\langle\nabla^{2}\rho(p)^{\alpha}v,v\rangle & = \alpha(\alpha-1)\rho(p)^{\alpha-2}|\langle \nabla\rho(p),v\rangle|^{2}+\alpha\rho(p)^{\alpha-1}\langle\nabla^{2}\rho(p)v,v\rangle \\ 
& \!\!\!\!\stackrel{\eqref{eq:comper}}{=} \alpha(\alpha-1)\rho(p)^{\alpha-2}\Big(\frac{v_{1}p_{1}}{2b}+h'(L-p_{2})v_{2}\Big)^{2} \\ 
& + \alpha\rho(p)^{\alpha-1} \Big(-\frac{1}{2b}v_{1}^{2} + h''(L-p_{2})v_{2}^{2} \Big) \eqqcolon \mathrm{I} + \mathrm{II}. 
\end{align*}
We distinguish two cases, and record in advance that $\mathrm{I},\mathrm{II}\leq 0$ because of $h''\leq 0$ by \ref{item:(iii)} and $\alpha<1$. First, suppose that $p_{1}=0$. Then $p\in K$ yields $h(L-p_{2})>0$ and so $0<L-p_{2}<S_{*}$ as noted in the preceding proof. If $h'(L-p_{2})\neq 0$ and $v_{2}\neq 0$, then $\alpha-1<0$ gives us $\mathrm{I}<0$, leading to $\mathrm{I}+\mathrm{II}<0$ because of $\mathrm{II}\leq 0$. Conversely, if $h'(L-p_{2})\neq 0$ and $v_{2}=0$, then still $v_{1}\neq 0$ and so we have $\mathrm{II}<0$ because of $h''\leq 0$ globally on $\R$. Since $\mathrm{I}\leq 0$, we arrive at $\mathrm{I}+\mathrm{II}<0$ in this case too. 

If $h'(L-p_{2})=0$, then $L-p_{2}=s_{*}$ and so $h''(L-p_{2})<0$ because of \ref{item:(iii)}, see also Figure \ref{fig:functionh}. Now, if $v_{1}=0$, then $v_{2}\neq 0$, and so we conclude $\mathrm{II}<0$ and therefore $\mathrm{I}+\mathrm{II}<0$ too. Lastly, if $v_{1}\neq 0$, then $\mathrm{II}<0$ because of $h''\leq 0$ and so $\mathrm{I}+\mathrm{II}<0$. 

Next, suppose that $p_{1}\neq 0$. If $v_{1}\neq 0$, we obtain $\mathrm{II}<0$ and so $\mathrm{I}+\mathrm{II}<0$. On the other hand, if $v_{1}=0$, then $v_{2}\neq 0$, and we again have two options: If $h'(L-p_{2})\neq 0$, we are done because of $\mathrm{I}<0$, and so $\mathrm{I}+\mathrm{II}<0$ in this case. Conversely, if $h'(L-p_{2})=0$, then $h''(L-p_{2})<0$ by \ref{item:(iii)} and so $\mathrm{II}<0$ together with $\mathrm{I}+\mathrm{II}<0$.

Summarizing, $\nabla^{2}\rho(p)^{\alpha}$ is negative definite regardless of $p\in K$, and so $\rho^{\alpha}$ is strictly concave. From here, it follows that $G_{0}|_{K}$ is strictly convex. Lastly, if $(p_{j})\subset K$ is a sequence with $p_{j}\to p\in\partial K$, then $\rho(p_{j})^{\alpha}\to 0$. As $G_{0}|_{K}\leq 0$, this implies the convexity of $G_{0}$, and the proof is complete. 
\end{proof}
Based on the preceding Lemma \ref{lem:convexity1}, we are allowed to form the Fenchel conjugate 
\begin{align}\label{eq:Ffinal}
\fbox{\text{$F\coloneqq G_{0}^{*}$,\;\;\;\text{that is},\;\;\;$F(z)=\sup_{p\in\overline{K}}z\cdot p - G_{0}(p)$,}} 
\end{align}
where we used that $G_{0}(p)=+\infty$ if $p\in\R^{2}\setminus\overline{K}$. We now proceed to show that $F$ is in fact $\mu$-elliptic and coincides with $F$ on certain cones $\mathcal{C}_{R,\delta}$. As a preparation, we single out the following two auxiliary results. 

\begin{figure}
\begin{tikzpicture}[scale=1]
\draw[->] (-4,0) -- (4,0); 
\draw[->] (0,-4) -- (0,2);
\draw[
    thick,
    color=blue!80!black,
    fill=blue!20!white, opacity=0.4
  ]
    plot[variable=\x, domain=1:-4, samples=200] 
      (
        { \x >= -1 ? 2*sqrt(1 - \x) : 2*sqrt(max(0, (1 - \x) - (5/9)*(-1 - \x)^2)) },
        { \x }
      )
    --
    plot[variable=\x, domain=-4:1, samples=200] 
      (
        { \x >= -1 ? -2*sqrt(1 - \x) : -2*sqrt(max(0, (1 - \x) - (5/9)*(-1 - \x)^2)) },
        { \x }
      )
    -- cycle;
    \draw[blue, thick, domain=-4:4, samples=100] 
        plot (\x, {1 - (\x)^2/4});
        \node[blue] at (4,-3.5) {$p_{2}=L-\frac{p_{1}^{2}}{4b}$};
        \draw[
    thick,
    color=blue!80!black,
    fill=blue!20!white, opacity=0.2, scale=0.8, yshift=-0.4cm
  ]
    plot[variable=\x, domain=1:-4, samples=200] 
      (
        { \x >= -1 ? 2*sqrt(1 - \x) : 2*sqrt(max(0, (1 - \x) - (5/9)*(-1 - \x)^2)) },
        { \x }
      )
    --
    plot[variable=\x, domain=-4:1, samples=200] 
      (
        { \x >= -1 ? -2*sqrt(1 - \x) : -2*sqrt(max(0, (1 - \x) - (5/9)*(-1 - \x)^2)) },
        { \x }
      )
    -- cycle;
    \node[blue!50!white] at (-2.5,-2) {\LARGE $K$};
    \node[blue!50!white] at (0,-2) {\LARGE $K\setminus K_{0}$};
    \node[blue!50!white] at (2.8,-2) {\Large $K_{0}$};
    \node[blue] at (0,1) {\textbullet};
     \node[above,blue] at (0.5,1) {$(0,L)$};
     \begin{scope}
    \clip (-1,0.745) [out=30, in = 180] to (0,1) [out=0, in =150] to (1,0.745) -- (1,-1) -- (-1,-1)--(-1,0.745);
    \draw[fill=green!60!black, opacity=0.3] (0,1) ellipse (0.5cm and 0.5cm);
    \node[green!50!black] at (0.25,0.5) {\LARGE $\mathfrak{L}$};
    \end{scope}
\end{tikzpicture}
\caption{The convex set $K$ (larger light blue set) and the set $K_{0}$ in the dual variable. The set $K_{0}$ is where large gradients are mapped to, and the dark green set $\mathfrak{L}$ schematically  corresponds to large gradients with dominating second component. Since we finally aim to place a jump in the $e_{2}$-direction (corresponding to the vertex $(0,L)$ in the stress  $\sigma=\nabla F(\nabla u)$) and use the concrete representation of $F$ via $F_{0}$ for the underlying singularity, we need to leave $F_{0}^{*}$ unchanged in $\mathfrak{L}$. The cut-off in the dual variable induced by $K$ finally leads to the linear growth of $F$, and it is then a key point of the proof that suitable choices of the cones from Figure \ref{fig:cone} are mapped to $\mathfrak{L}$.}
\label{fig:mod}
\end{figure}
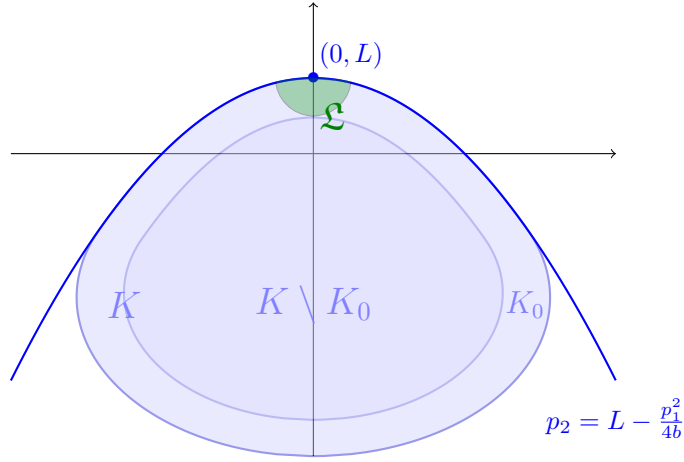
\begin{lemma}\label{lem:diffeo}
Let $K\subset\R^{2}$ be open and bounded, and let $\Phi\colon\R^{2}\to K$ be continuous and bijective with continuous inverse. Then, for any $\varepsilon>0$, there exists a number $\mathtt{R}>\frac{1}{\varepsilon}$ such that 
\begin{align}
\sup_{|z|\geq\mathtt{R}}\mathrm{dist}(\Phi(z),\partial K)<\varepsilon. 
\end{align}
\end{lemma}
\begin{proof}
Suppose that the claim is false. Then there exist $\varepsilon>0$ and $j_{0}\in\mathbb{N}$ such that, for any $j\geq j_{0}$, there exists $z_{j}\in\R^{2}$ with $|z_{j}|\geq j$ and $\mathrm{dist}(\Phi(z_{j}),\partial K)>\frac{\varepsilon}{2}$. Defining $\mathcal{K}\coloneqq \{y\in K\colon\;\mathrm{dist}(y,\partial K)\geq \frac{\varepsilon}{2}\}$, $\mathcal{K}$ is compact and contained in $K$, and so is 
\begin{align*}
\mathcal{K}'\coloneqq \overline{\{\Phi(z_{j})\colon\;j\in\mathbb{N}\}}.
\end{align*}
In particular, since $\Phi^{-1}\colon K\to\R^{2}$ is continuous, $\Phi^{-1}(\mathcal{K}')$ is compact and, thus, is  bounded. This yields that $(z_{j})=(\Phi^{-1}(\Phi(z_{j})))$ is bounded, a contradiction. The proof is complete. 
\end{proof}
\begin{lemma}\label{lem:eigenvalues}
Let $\mathcal{K}\subset\R^{2}$ be compact, and let $H\colon \mathcal{K}\to\R_{\mathrm{sym}}^{2\times 2}$ be continuous. Then the maps 
\begin{align*}
\mathscr{H}^{-}\colon \mathcal{K}\ni x \mapsto \lambda_{H(x)}^{-}\in\R\;\;\;\text{and}\;\;\;\mathscr{H}^{+}\colon \mathcal{K}\ni x\mapsto \lambda_{H(x)}^{+}\in\R
\end{align*}
are continuous. In particular, if $H(x)$ is moreover positive definite for every $x\in\mathcal{K}$, then 
\begin{align*}
0<\min_{x\in\mathcal{K}}\lambda_{H(x)}^{-}\leq \max_{x\in\mathcal{K}}\lambda_{H(x)}^{+}<\infty.
\end{align*}
\end{lemma}
\begin{proof}
By the Rayleigh characterization of eigenvalues, we have for all $x,y\in\mathcal{K}$ and all $v\in\mathbb{S}^{1}$ that 
\begin{align*}
\langle H(x)v,v\rangle = \langle (H(x)-H(y))v,v\rangle + \langle H(y)v,v\rangle \geq - \|H(x)-H(y)\| + \lambda_{H(y)}^{-}, 
\end{align*}
and so minimizing the overall inequality over $v\in\mathbb{S}^{1}$ gives us 
\begin{align}
\lambda_{H(y)}^{-}-\lambda_{H(x)}^{-}\leq \|H(x)-H(y)\|. 
\end{align}
Interchanging the roles of $x$ and $y$, we obtain 
\begin{align}
\lambda_{H(x)}^{-}-\lambda_{H(y)}^{-}\leq \|H(x)-H(y)\|\;\;\;\text{and so}\;\;\; |\lambda_{H(x)}^{-}-\lambda_{H(y)}^{-}|\leq \|H(x)-H(y)\|. 
\end{align}
Hence, $\mathscr{H}^{-}$ is continuous. The argument for $\mathscr{H}^{+}$ is analogous, and so the proof is complete. 
\end{proof}
We now have: 
\begin{proposition}[$\hold^{2}$-regularity and $\mu$-ellipticity of $F$]\label{prop:Fmuell}
Let $\mu>2$. The function $F$ from \eqref{eq:Ffinal} is of class $\hold^{2}(\R^{2})$, is of linear growth and is \emph{globally $\mu$-elliptic}.  
\end{proposition}
\begin{proof}
The linear growth of $F$ directly follows from \eqref{eq:Ffinal} and the fact that $K$ is bounded. For the remaining assertions, we split the proof into three steps. 

\emph{Step 1.} Our first objective, to be executed in Steps 1 and 2, is the $\hold^{2}$-regularity of $F$. 
To this end, let $z\in\R^{2}$ be arbitrary and define  $f_{0}^{z}\colon \overline{K}\ni p \mapsto z\cdot p -G_{0}(p)$. By continuity of $G_{0}|_{\overline{K}}$  and since $\overline{K}$ is compact by Lemma \ref{lem:Kprops}\ref{item:K1}, $f_{0}^{z}$ has a maximizer $p_{z}$ on $\overline{K}$. We claim that any such  maximizer necessarily belongs to $K$. To this end, we argue by contradiction and suppose that $p_{z}\in\partial K$. Since $\partial K$ is smooth by Lemma \ref{lem:Kprops}\ref{item:K3}, the outer normal $\nu_{\partial K}(p_{z})$ to $\partial K$ at $p_{z}$ is well-defined. For a sufficiently small $t_{0}>0$, we consider the line segment 
\begin{align*}
l \coloneqq \{\mathbf{p}_{z}(t)\coloneqq p_{z}-t\nu_{\partial K}(p_{z})\colon\;0\leq t\leq t_{0}\} =\Big\{\mathbf{p}_{z}(t)\coloneqq p_{z}+t\frac{\nabla\rho(p_{z})}{|\nabla \rho(p_{z})|}\colon\;0\leq t\leq t_{0} \Big\}, 
\end{align*}
which we may assume to be contained in $K$ except for the point $p_{z}$. Now, since $p_{z}\in\partial K$, $G_{0}(p_{z})=0$ and therefore 
\begin{align*}
f_{0}^{z}(\mathbf{p}_{z}(t)) & = z\cdot \mathbf{p}_{z}(t) - G_{0}(\mathbf{p}_{z}(t)) \\ 
& = z \cdot \Big(p_{z} + t\frac{\nabla\rho(p_{z})}{|\nabla\rho(p_{z})|} \Big) + C(A,\mu)\rho\Big(p_{z} + t\frac{\nabla\rho(p_{z})}{|\nabla\rho(p_{z})|} \Big)^{\alpha} \\ 
& = z\cdot p_{z} - G_{0}(p_{z}) \\ 
& + tz\cdot \frac{\nabla\rho(p_{z})}{|\nabla\rho(p_{z})|} + C(A,\mu)\rho\Big(p_{z} + t\frac{\nabla\rho(p_{z})}{|\nabla\rho(p_{z})|} \Big)^{\alpha} \\ 
& = (\max_{p\in \overline{K}}f_{0}^{z}(p)) + \Big(tz\cdot \frac{\nabla\rho(p_{z})}{|\nabla\rho(p_{z})|} + C(A,\mu)\rho\Big(p_{z} + t\frac{\nabla\rho(p_{z})}{|\nabla\rho(p_{z})|} \Big)^{\alpha}\Big) \\ 
& \eqqcolon (\max_{p\in \overline{K}}f_{0}^{z}(p)) + \gamma(t). 
\end{align*}
Hence, we shall arrive at a contradiction provided we can show $\gamma(t)>0$ for some $0<t<t_{0}$. Since $\rho$ is differentiable at $p_{z}$ and $\rho(p_{z})=0$, there exists a function $\mathrm{r}$ with $\lim_{t\searrow 0}|\frac{\mathrm{r}(t)}{t}|=0$ such that 
\begin{align*}
\rho(p_{z}-t\nu_{\partial K}(p_{z})) = \rho(p_{z}) +(\nabla\rho(p_{z}))\cdot(-t\nu_{\partial K}(p_{z}))+\mathrm{r}(t) = t|\nabla\rho(p_{z})| + \mathrm{r}(t)
\end{align*}
for all $0<t<t_{0}$. Diminishing $t_{0}$ if required, we may assume that $|\mathrm{r}(t)|\leq \frac{\varepsilon_{0}}{2}t$ for all $0<t<t_{0}$, where $\varepsilon_{0}>0$ is as in \eqref{eq:boundobound}. For such values of $t$, we have 
\begin{align*}
\rho(p_{z}-t\nu_{\partial K}(p_{z}))\geq \varepsilon_{0} t - \frac{\varepsilon_{0}}{2}t = \frac{\varepsilon_{0}}{2}t
\end{align*}
and therefore 
\begin{align*}
\gamma(t) & \geq t^{\alpha}\Big(t^{1-\alpha}z\cdot\frac{\nabla \rho(p_{z})}{|\nabla\rho(p_{z})|} + C(A,\mu)\Big(\frac{\varepsilon_{0}}{2}\Big)^{\alpha}\Big)\qquad\text{for all}\;0<t<t_{0}. 
\end{align*}
Because of $0<\alpha<1$ and $\lim_{t\searrow 0}t^{1-\alpha}=0$, the term inside the brackets will be positive 
for some sufficiently small $0<t<t_{0}$. For this choice of $t$, we particularly have $\gamma(t)>0$, and so arrive at the contradictory estimate $f_{0}^{z}(\mathbf{p}_{z}(t))> \max_{p\in\overline{K}}f_{0}^{z}(p)$. From now on, we may thus assume that any maximizer of $f_{0}^{z}$ belongs to $K$. 

\emph{Step 2. $\hold^{2}$-regularity of $F$.} We now draw some elementary conclusions that are well known facts in convex analysis (see, e.g., \cite{Rockafellar1970}); for the sake of exposition, we provide the details. As a first consequence of Step 1, the maximizer $p_{z}$ of $f_{0}^{z}$ is unique. Indeed, $p_{z}$ belongs to the open set $K$ where, by Lemma \ref{lem:convexity1}, $f_{0}^{z}$ is strictly concave. This particularly entails that the map 
\begin{align*}
\Phi\colon \R^{2}\ni z \mapsto p_{z}\in K 
\end{align*}
is well-defined and 
\begin{align}\label{eq:weizen}
\nabla f_{0}^{z}(\Phi(z))=0,\;\;\;\text{that is},\;\;\;z=(\nabla G_{0})(\Phi(z)).
\end{align}
From \eqref{eq:weizen}, we deduce that $\nabla G_{0}\colon K\to \R^{2}$ is surjective. By openness and convexity of $K$, and since $G_{0}|_{K}\in\hold^{2}(K)$ is strictly convex with $\nabla^{2}G_{0}$ being positive definite everywhere in $K$, $\nabla G_{0}\colon K\to\R^{2}$ is injective and $\det(\nabla^{2}G_{0})\neq 0$ globally on $K$. Thus $\nabla G_{0}\colon K\to\R^{2}$ is bijective, of class $\hold^{1}$ and, by the inverse function theorem, $(\nabla G_{0})^{-1}\colon \R^{2}\to K$ is of class $\hold^{1}$ too. Going back to \eqref{eq:weizen}, we thus obtain that 
\begin{align}\label{eq:PhiReg}
\Phi =(\nabla G_{0})^{-1}\colon \R^{2} \to K\qquad\text{is of class $\hold^{1}$ too}. 
\end{align}
Hence, we have for $z\in\R^{2}$ that 
\begin{align*}
F(z)=f_{0}^{z}(\Phi(z))=z\cdot\Phi(z)-G_{0}(\Phi(z)), 
\end{align*}
from where it follows that $F\in\hold^{1}(\R^{2})$. This allows us to compute  
\begin{align}\label{eq:simplifier1}
\nabla F(z) & = \nabla(f_{0}^{z}(\Phi(z))) =\nabla \Big(z\cdot \Phi(z)-G_{0}(\Phi(z))\Big) \\ & = \Phi(z) + (\nabla \Phi(z))^{\top}z-(\nabla\Phi(z))^{\top}(\nabla G_{0})(\Phi(z)) \stackrel{\eqref{eq:weizen}}{=} \Phi(z)
\end{align}
for all $z\in\R^{2}$. Due to this identity and \eqref{eq:PhiReg}, we infer that $F\in\hold^{2}(\R^{2})$. 

\emph{Step 3. $\mu$-ellipticity of $F$.} Let $z\in\R^{2}$. Then  \eqref{eq:weizen} and \eqref{eq:simplifier1} combine to 
\begin{align}
z=\nabla G_{0}(\Phi(z))=(\nabla G_{0})(\nabla F(z))\;\;\;\text{and so}\;\;\;\mathbbm{1}_{2\times 2}= (\nabla^{2}G_{0})(\nabla F(z))\nabla^{2}F(z).
\end{align}
We then arrive at the key identity 
\begin{align}\label{eq:keyidentity}
\nabla^{2}F(z)=((\nabla^{2}G_{0})(\Phi(z)))^{-1}\qquad\text{for all}\;z\in\R^{2}.
\end{align}
Based on \eqref{eq:keyidentity}, we now proceed to show that $F$ is $\mu$-elliptic.  For sufficiently small $\ell>0$, let $\Sigma_{\ell}\coloneqq \rho^{-1}(\{\ell\})\Subset K_{0}$. The definition of $K_{0}$ then yields that $\Sigma_{\ell}$ is a smooth one-dimensional manifold. Let $p\in\Sigma_{\ell}$. By slight abuse of notation, we now put  
\begin{align}\label{eq:basis}
\nu_{p}\coloneqq \frac{\nabla\rho(p)}{|\nabla\rho(p)|},
\end{align}
so that $\nu_{p}$ is normal to $\Sigma_{\ell}$, and define a corresponding tangential vector via
\begin{align}\label{eq:basis1}
\tau_{p}\coloneqq \frac{1}{|\nabla\rho(p)|}\left(\begin{matrix} -\partial_{2}\rho(p) \\ \partial_{1}\rho(p) \end{matrix} \right).
\end{align}
Now, $\{\nu_{p},\tau_{p}\}$ forms an orthonormal basis of $\R^{2}$, and so the representing matrix of 
\begin{align*}
\nabla^{2}G_{0}(p) & \stackrel{\eqref{eq:miami}}{=} -C(A,\mu)\alpha(\alpha-1)\rho(p)^{\alpha-2}\nabla\rho(p)\otimes\nabla\rho(p) - C(A,\mu)\alpha\rho(p)^{\alpha-1}\nabla^{2}\rho(p) \\ & \;\;\eqqcolon \mathscr{A}(p)+\mathscr{B}(p)
\end{align*}
with respect to the basis specified in \eqref{eq:basis} and \eqref{eq:basis1} is given by 
\begin{align}\label{eq:type0negative}
\begin{split}
\mathfrak{A}(p)& \coloneqq \left(\begin{matrix} \langle\nabla^{2}G_{0}(p)\nu_{p},\nu_{p}\rangle & \langle\nabla^{2}G_{0}(p)\nu_{p},\tau_{p}\rangle \\ \langle\nabla^{2}G_{0}(p)\nu_{p},\tau_{p}\rangle & \langle\nabla^{2}G_{0}(p)\tau_{p},\tau_{p}\rangle \end{matrix} \right) \\ 
& = \left(\begin{matrix} C(A,\mu)\alpha(1-\alpha)\rho(p)^{\alpha-2}|\nabla\rho(p)|^{2} & 0 \\0 & 0\end{matrix} \right) + \left(\begin{matrix} \langle\mathscr{B}(p)\nu_{p},\nu_{p} \rangle& \langle \mathscr{B}(p)\nu_{p},\tau_{p}\rangle\\ \langle\mathscr{B}(p)\nu_{p},\tau_{p}\rangle & \langle\mathscr{B}(p)\tau_{p},\tau_{p}\rangle\end{matrix} \right) \\ 
& \eqqcolon \left(\begin{matrix} C(A,\mu)\alpha(1-\alpha)\rho(p)^{\alpha-2}|\nabla\rho(p)|^{2} & 0 \\0 & 0\end{matrix} \right) + \left(\begin{matrix} \mathscr{B}_{\nu\nu}(p) & \mathscr{B}_{\nu\tau}(p)\\ \mathscr{B}_{\tau\nu}(p) & \mathscr{B}_{\tau\tau}(p)\end{matrix} \right).
\end{split}
\end{align}
By Lemma \ref{lem:convexity1}, $\mathfrak{A}(p)\in\R^{2\times 2}$ is a positive definite symmetric matrix. Based on the notation specified in \S\ref{sec:notation}, the Rayleigh characterization of eigenvalues yields
\begin{align}\label{eq:ellipto1}
\begin{split}
\lambda_{\mathfrak{A}(p)}^{+} & =\max_{\substack{\xi\in\R^{2} \\ |\xi|=1}}\langle\mathfrak{A}(p)\xi,\xi\rangle \geq \langle\mathfrak{A}(p)e_{1}, e_{1}\rangle\\ 
& = C(A,\mu)\alpha(1-\alpha)\rho(p)^{\alpha-2}|\nabla\rho(p)|^{2} + \mathscr{B}_{\nu\nu}(p) \\ 
& \!\!\!\!\! \stackrel{\eqref{eq:comper}_{2}}{\geq} C(A,\mu)\alpha(1-\alpha)\rho(p)^{\alpha-2} |\nabla\rho(p)|^{2} \\ & + \underbrace{C(A,\mu)\alpha\frac{\rho(p)^{\alpha-1}}{|\nabla\rho(p)|^{2}}\Big(\frac{1}{2b}(\partial_{1}\rho(p))^{2} -h''(L-p_{2})(\partial_{2}\rho(p))^{2} \Big)}_{\geq 0\,\text{by concavity of $h$}} \\ 
& \!\!\!\stackrel{\eqref{eq:boundobound}}{\geq} C(A,\mu)\alpha(1-\alpha)\varepsilon_{0}^{2}\rho(p)^{\alpha-2}. 
\end{split}
\end{align}
For the corresponding upper bound, we record that 
\begin{align}\label{eq:twotribes}
\begin{split}
|\mathscr{B}(p)| & \stackrel{\eqref{eq:comper}_{2}}{\leq}  C(A,\mu)\alpha\rho(p)^{\alpha-1}\Big(\frac{1}{4b^{2}}+\sup_{s\in[0,S_{*}]}|h''(s)|^{2}\Big)^{\frac{1}{2}} \\ 
&  \;\stackrel{\eqref{eq:boundobound}}{\leq} C(A,\mu)\alpha\rho(p)^{\alpha-2}\Big(\frac{1}{4b^{2}}+\sup_{s\in[0,S_{*}]}|h''(s)|^{2}\Big)^{\frac{1}{2}} \eqqcolon C'\rho(p)^{\alpha-2}, 
\end{split}
\end{align}
where $C'=C'(A,b,\mu,h)>0$ is a constant. On the other hand, we have 
\begin{align}\label{eq:bettedavieseyes}
\begin{split}
\lambda_{\mathfrak{A}(p)}^{+} & \leq \lambda_{\mathfrak{A}(p)}^{+} + \lambda_{\mathfrak{A}(p)}^{-} =\mathrm{Tr}(\mathfrak{A}(p)) \\ & \!\!\!\!\stackrel{\eqref{eq:type0negative}}{\leq}  C(A,\mu)\alpha(1-\alpha)\rho(p)^{\alpha-2}|\nabla\rho(p)|^{2} + \mathscr{B}_{\nu\nu}(p)+\mathscr{B}_{\tau\tau}(p) \\ 
& \!\!\!\!\!\!\!\!\!\!\!\stackrel{\eqref{eq:boundobound},\,\eqref{eq:twotribes}}{\leq} C''\rho(p)^{\alpha-2}, 
\end{split}
\end{align}
where again $C''=C''(A,b,\mu,h,\varepsilon_{0},\Theta_{0})>0$ is a constant. Combining \eqref{eq:ellipto1} and \eqref{eq:bettedavieseyes}, there exist positive constants $c_{1},c_{2}>0$, both depending on $A,b,\mu,h,\varepsilon_{0},\Theta_{0}$, such that 
\begin{align}\label{eq:uppereigenvalue}
c_{1}\rho(p)^{\alpha-2}\leq \lambda_{\mathfrak{A}(p)}^{+}\leq c_{2}\rho(p)^{\alpha-2} \qquad\text{for all}\;p\in K_{0}. 
\end{align}
Next, we bound $\det(\mathfrak{A}(p))$ from above and below. We have 
\begin{align}\label{eq:detlefjoseph}
\begin{split}
\det(\mathfrak{A}(p)) & = \big(C(A,\mu)\alpha(1-\alpha)\rho(p)^{\alpha-2}|\nabla\rho(p)|^{2}+\mathscr{B}_{\nu\nu}(p)\Big)\mathscr{B}_{\tau\tau}(p) - \mathscr{B}_{\nu\tau}(p)^{2} \\ 
& = C(A,\mu)\alpha(1-\alpha)\rho(p)^{\alpha-2}|\nabla\rho(p)|^{2}\mathscr{B}_{\tau\tau}(p) + \det(\mathscr{B}(p)) \\ 
& \geq C(A,\mu)\alpha(1-\alpha)\rho(p)^{\alpha-2}|\nabla\rho(p)|^{2}\mathscr{B}_{\tau\tau}(p), 
\end{split}
\end{align}
where we used that $\det(\mathscr{B}(p))\geq 0$, see  $\eqref{eq:comper}_{2}$ and recall that $h$ is concave. To proceed further, we require an explicit representation of $\mathscr{B}_{\tau\tau}(p)$: 
\begin{align}\label{eq:carribeanqueen}
\begin{split}
\mathscr{B}_{\tau\tau}(p) & \stackrel{\eqref{eq:comper}}{=} - \frac{C(A,\mu)\alpha\rho(p)^{\alpha-1}}{|\nabla\rho(p)|^{2}}\left(\left(\begin{matrix} -\frac{1}{2b} & 0 \\ 0 & h''(L-p_{2})\end{matrix} \right)\left(\begin{matrix} -\partial_{2}\rho(p) \\ \partial_{1}\rho(p)\end{matrix}\right) \right)\cdot\left(\begin{matrix} -\partial_{2}\rho(p) \\ \partial_{1}\rho(p)\end{matrix} \right) \\ 
& =  - {C(A,\mu)\alpha\rho(p)^{\alpha-1}}\frac{\Big(-\frac{1}{2b}(\partial_{2}\rho(p))^{2}+h''(L-p_{2})(\partial_{1}\rho(p))^{2} \Big)}{|\nabla\rho(p)|^{2}} \\ 
& =   {C(A,\mu)\alpha\rho(p)^{\alpha-1}}\frac{\frac{1}{2b}(h'(L-p_{2}))^{2}-\frac{1}{b}h''(L-p_{2})(\frac{p_{1}^{2}}{4b}) }{|\nabla\rho(p)|^{2}}\\ 
& \eqqcolon  {C(A,\mu)\alpha\rho(p)^{\alpha-1}} \mathrm{III}(p).
\end{split}
\end{align}
We firstly examine the term $\mathrm{III}$ at $\partial K$, that is, where $\rho(p)=0$ and so $h(L-p_{2})=\frac{p_{1}^{2}}{4b}$. Based on the latter relation, we may write 
\begin{align*}
  \mathrm{III}(p) & = \frac{1}{b}  \frac{\frac{1}{2}(h'(L-p_{2}))^{2}-h''(L-p_{2})h(L-p_{2})}{|\nabla\rho(p)|^{2}}. 
\end{align*}
By \eqref{eq:boundobound}, it suffices to give a lower bound on the enumerator. Since $h(L-p_{2})=\frac{p_{1}^{2}}{4b}\geq 0$, the concavity of $h$ yields that both summands in the enumerator are non-negative. Now suppose that $p\in\partial K$ is such that the enumerator is zero. Then $h'(L-p_{2})=0$ and so $L-p_{2}=s_{*}$, but then $h''(s_{*})<0$ and $h(s_{*})>0$. Hence, the enumerator cannot be zero on $\partial K$. By compactness of $\partial K$, continuity of $\partial K\ni p\mapsto \mathrm{III}(p)$ and recalling \eqref{eq:boundobound}, we thus have
\begin{align*}
    \mathtt{c}\coloneqq \min_{p\in\partial K}\mathrm{III}(p)>0.
\end{align*}
By continuity of $p\mapsto\mathrm{III}(p)$, there exists an open set $K_{1}\subset K_{0}$ with $\partial K= \partial K_{1}\cap \partial K$ such that 
\begin{align}
\mathrm{III}(p)\geq \frac{\mathtt{c}}{2}\qquad\text{for all}\; p\in K_{1}.
\end{align}
Going back to \eqref{eq:carribeanqueen}, it thus follows that 
\begin{align*}
\mathscr{B}_{\tau\tau}(p)\geq \frac{\alpha \mathtt{c}C(A,\mu)}{2}\rho(p)^{\alpha-1}\qquad\text{for all}\;p\in K_{1}
\end{align*}
and hence 
\begin{align}\label{eq:mutualattraction}
\det(\mathfrak{A}(p))\stackrel{\eqref{eq:detlefjoseph}}{\geq} \frac{\mathtt{c}\varepsilon_{0}^{2}}{2}C(A,\mu)^{2}\alpha^{2}(1-\alpha)\rho(p)^{2\alpha-3}\geq \mathtt{C}(A,\alpha,\mu,\varepsilon_{0},K_{1})\rho(p)^{2\alpha-3}, 
\end{align}
where $\mathtt{C}(A,\alpha,\mu,\varepsilon_{0},K_{1})>0$ is a constant. Clearly, because of $K_{1}\subset K_{0}$, \eqref{eq:uppereigenvalue} holds true on $K_{1}$. In conclusion, 
\begin{align}\label{eq:smallerinho}
    \lambda_{\mathfrak{A}(p)}^{-} = \frac{\det(\mathfrak{A}(p))}{\lambda_{\mathfrak{A}(p)}^{+}} 
    \stackrel{\eqref{eq:uppereigenvalue},\,\eqref{eq:mutualattraction}}{\geq} \frac{\mathtt{C}(A,\alpha,\mu,\varepsilon_{0},K_{1})\rho(p)^{2\alpha-3}}{c_{2}\rho(p)^{\alpha-2}} \eqqcolon c_{0}\rho(p)^{\alpha-1}
\end{align}
holds for all $p\in K_{1}$. Hence, returning to \eqref{eq:uppereigenvalue}, \eqref{eq:smallerinho} gives us
\begin{align}\label{eq:fernando}
    c_{0}\rho(p)^{\alpha-1}\leq \lambda_{\mathfrak{A}(p)}^{-} \leq \lambda_{\mathfrak{A}(p)}^{+} \leq c_{2}\rho(p)^{\alpha-2}\qquad\text{for all}\;p\in K_{1}.
\end{align}
We are now ready to conclude the proof. Namely, we let $\varepsilon>0$ be so small such that 
\begin{align*}
K^{\varepsilon}\coloneqq \{p\in K\colon\;\mathrm{dist}(p,\partial K)<\varepsilon\}\subset K_{1}. 
\end{align*}
By what we have established between \eqref{eq:weizen} and \eqref{eq:PhiReg}, $\Phi\colon\R^{2}\to K$ is a $\hold^{1}$-diffeomorphism. In this situation, Lemma \ref{lem:diffeo} provides us with some $\mathtt{R}>1$ such that 
\begin{align}\label{eq:miamivice}
\sup_{|z|\geq \mathtt{R}}\mathrm{dist}(\Phi(z),\partial K)<\varepsilon. 
\end{align}
In particular, if $|z|\geq \mathtt{R}$, then $\Phi(z)\in K^{\varepsilon}\subset K_{1}$. Hence, if $|z|\geq\mathtt{R}$, then $\Phi(z)\in K_{1}(\subset K_{0})$, and we have 
\begin{align}\label{eq:tony}
\begin{split}
|z| & \stackrel{\eqref{eq:weizen}}{=} |(\nabla G_{0})(\Phi(z))| =C(A,\mu)\alpha|\rho(\Phi(z))|^{\alpha-1}|\nabla\rho(\Phi(z))| \\ & \!\!\!\!\!\!\!\!\!\!\!\stackrel{\Phi(z)\in K_{1},\,\eqref{eq:boundobound}}{\leq} C(A,\mu)\Theta_{0}\alpha|\rho(\Phi(z))|^{\alpha-1}.
\end{split}
\end{align}
Based on \eqref{eq:keyidentity} and the fact that eigenvalues are independent of the particular choice of bases, we thus find for any such $z$ that 
\begin{align}\label{eq:upperinho}
\begin{split}
\lambda_{\nabla^{2}F(z)}^{+} & \stackrel{\eqref{eq:keyidentity}}{=}  \big(\lambda_{(\nabla^{2}G_{0}(\Phi(z)))}^{-}\big)^{-1} = \big(\lambda_{\mathfrak{A}(\Phi(z))}^{-}\big)^{-1} \\ & \stackrel{\eqref{eq:fernando}}{\leq} \frac{1}{c_{0}}\rho(\Phi(z))^{1-\alpha} \stackrel{\eqref{eq:tony}}{=} \frac{C(A,\mu)\Theta_{0}\alpha}{c_{3}}\frac{1}{|z|} \\ & \stackrel{|z|\geq 1}{\leq} \frac{2C(A,\mu)\Theta_{0}\alpha}{c_{3}}\frac{1}{1+|z|} \eqqcolon \widetilde{\Lambda}(1+|z|)^{-1}.
\end{split}
\end{align}
For the reverse estimate, we use \eqref{eq:boundobound} and $K^{\varepsilon}\subset K_{1}\subset K_{0}$ to replace \eqref{eq:tony} by 
\begin{align}\label{eq:tony1}
|z| \geq C(A,\mu)\alpha\varepsilon_{0}|\rho(\Phi(z))|^{\alpha-1}. 
\end{align}
Then, using that 
\begin{align*}
\frac{2-\alpha}{1-\alpha}=\frac{2-\frac{\mu-2}{\mu-1}}{1-\frac{\mu-2}{\mu-1}}=\frac{2\mu-2-\mu+2}{\mu-1-\mu+2} = \mu, 
\end{align*}
we conclude that 
\begin{align*}
\rho(\Phi(z))^{2-\alpha} & = \Big(\rho(\Phi(z))^{1-\alpha}\Big)^{\frac{2-\alpha}{1-\alpha}}  \stackrel{\eqref{eq:tony1}}{\geq} \Big(C(A,\mu)\alpha\varepsilon_{0}|z|^{-1} \Big)^{\mu} \geq  (C(A,\mu)\alpha\varepsilon_{0})^{\mu}(1+|z|)^{-\mu} 
\end{align*}
for all $|z|\geq \mathtt{R}$. In consequence, we find similarly as in \eqref{eq:upperinho} that 
\begin{align}\label{eq:donkey}
\begin{split}
    \lambda_{\nabla^{2}F(z)}^{-}  = \big(\lambda_{\mathfrak{A}(\Phi(z))}^{+}\big)^{-1} & \stackrel{\eqref{eq:fernando}}{\geq} \frac{1}{c_{2}}\rho(\Phi(z))^{2-\alpha} \\ 
    & \;\;\geq \frac{(C(A,\mu)\alpha\varepsilon_{0})^{\mu}}{c_{2}}(1+|z|)^{-\mu} \eqqcolon \widetilde{\lambda}(1+|z|)^{-\mu}. 
    \end{split}
\end{align}
To conclude the proof, we note that $\Phi\colon\R^{2}\to K$ is a $\hold^{1}$-diffeomorphism and thus maps the compact set $\{|z|\leq\mathtt{R}\}$ to a compact set $K''\subset K$. Hence, if $|z|\leq \mathtt{R}$, then  
\begin{align}\label{eq:donkey1}
\lambda_{\nabla^{2}F(z)}^{-} = \Big(\lambda_{\nabla^{2}G_{0}(\Phi(z))}^{+}\Big)^{-1} \geq \Big(\sup_{p\in K''}\lambda_{\nabla^{2}G_{0}(p)}^{+}\Big)^{-1} \geq c(K'')>0
\end{align}
and 
\begin{align}\label{eq:donkey2}
\lambda_{\nabla^{2}F(z)}^{+} = \Big(\lambda_{\nabla^{2}G_{0}(\Phi(z))}^{-}\Big)^{-1} \leq \Big(\inf_{p\in K''}\lambda_{\nabla^{2}G_{0}(p)}^{-}\Big)^{-1} \leq c'(K'')<\infty
\end{align}
by Lemmas \ref{lem:convexity1} and  \ref{lem:eigenvalues}. Hence, \eqref{eq:upperinho}, \eqref{eq:donkey},  \eqref{eq:donkey1} and \eqref{eq:donkey2}  imply that $F$ is $\mu$-elliptic on the entire $\R^{2}$. The proof is complete. 
\end{proof}
We finish this section by showing that $F$ coincides with $F_{0}$ on suitably chosen cones. 
\begin{proposition}\label{prop:coincidence}
Let $\mu>2$. There exists $\delta_{0}>0$ such that, for any $0<\delta<\delta_{0}$, there exists $R=R(\delta)>1$ such that the following hold: 
\begin{enumerate}
\item\label{item:coinc1} For $z\in\mathcal{C}_{R,\delta}$, we have $\nabla F_{0}(z)\in K$.
\item\label{item:coinc2} We have $F=F_{0}$  in $\mathcal{C}_{R,\delta}$.
\end{enumerate}
Here, $\mathcal{C}_{R,\delta}$ is as in \eqref{eq:cones}. 
\end{proposition}
\begin{proof}
Before we embark on the actual proof, we recall that 
\begin{align}\label{eq:inc}
K\subset \Big\{p\in\R^{2}\colon\;L-\frac{p_{1}^{2}}{4b}-p_{2}>0\Big\}, 
\end{align}
which follows directly from 
\begin{align*}
p\in K & \Longrightarrow h(L-p_{2})-\frac{p_{1}^{2}}{4b}>0 \stackrel{\text{\ref{item:(i)}}}{\Longrightarrow} L-p_{2}-\frac{p_{1}^{2}}{4b} >0.
\end{align*}
Towards the claim of the proposition, we put 
\begin{align*}
    \delta_{0}\coloneqq \Big(\frac{L}{b}\Big)^{\frac{1}{2}}
\end{align*}
and let $0<\delta<\delta_{0}$ be arbitrary. Because of $\mu>2$, we may subsequently choose $R=R(\delta)>0$ so large such that 
\begin{align}\label{eq:michaeljackson}
    A(\mu-2)R^{1-\mu}+b\delta^{2}\stackrel{!}{<}L<2L. 
\end{align}
For this choice of $R$, let $z\in\mathcal{C}_{R,\delta}$. We abbreviate $p=\nabla F_{0}(z)$, whereby \eqref{eq:firstvariation} yields 
\begin{align*}
L-p_{2} \stackrel{\eqref{eq:firstvariation}}{=}A(\mu-2)z_{2}^{1-\mu} + b\frac{z_{1}^{2}} {z_{2}^{2}} \stackrel{z\in\mathcal{C}_{R,\delta}}{\leq} A(\mu-2)R^{1-\mu} + b\delta^{2} \stackrel{\eqref{eq:michaeljackson}}{<}L
\end{align*}
because of $\mu>2$, and 
\begin{align*}
p_{2}= L-A(\mu-2)z_{2}^{1-\mu}-b\frac{z_{1}^{2}}{z_{2}^{2}} >L-A(\mu-2)R^{1-\mu}-b\delta^{2}\stackrel{\eqref{eq:michaeljackson}}{>}0.
\end{align*}
By \ref{item:(i)} in the construction of $h$, see also Figure \ref{fig:functionh}, we thus have $h(L-p_{2})=L-p_{2}$. In particular, 
\begin{align}\label{eq:debrahmorgan}
\begin{split}
\rho(p)& =h(L-p_{2})-\frac{p_{1}^{2}}{4b}=L-p_{2}-\frac{p_{1}^{2}}{4b} 
\\ 
 & = A(\mu-2)z_{2}^{1-\mu} + b\frac{z_{1}^{2}}{z_{2}^{2}} -b\frac{z_{1}^{2}}{z_{2}^{2}} =(\mu-2)Az_{2}^{1-\mu}\stackrel{z\in\mathcal{C}_{R,\delta},\,\mu>2}{>}0. 
 \end{split}
\end{align}
Hence, $p=\nabla F_{0}(z)\in K$, which is \ref{item:coinc1}. This allows us to conclude that 
\begin{align}\label{eq:onlythelonely}
\begin{split}
\widetilde{F}_{0}^{*}(p)  & \stackrel{\text{Rem. \ref{rem:extended}}}{=} F_{0}^{*}(p) \stackrel{\eqref{eq:wildestdreams},\,p\in K,\,\eqref{eq:inc}}{=} -C(A,\mu)\Big(L-\frac{p_{1}^{2}}{4b}-p_{2}\Big)^{\alpha} \\ & \;\,\stackrel{\eqref{eq:debrahmorgan}_{1}}{=} -C(A,\mu)\rho(p)^{\alpha} \stackrel{\eqref{eq:swallowtears},\,p\in K}{=} G_{0}(p).
\end{split}
\end{align}
From here, the assertion of the proposition essentially follows from routine duality arguments; we give the full details. With the extended real-valued integrand $\widetilde{F}_{0}$ from \eqref{eq:extended}, we have  
\begin{align}\label{eq:moon1}
\begin{split}
F(z) & \stackrel{\eqref{eq:Ffinal}}{=} \sup_{p\in \overline{K}}z\cdot p - G_{0}(p) \stackrel{\nabla F_{0}(z)\in K}{\geq} z\cdot\nabla F_{0}(z) - G_{0}(\nabla F_{0}(z))  \\ & \stackrel{\eqref{eq:onlythelonely}}{=} z\cdot\nabla F_{0}(z) - \widetilde{F}_{0}^{*}(\nabla F_{0}(z)) \stackrel{(*)_{1}}{=}  \widetilde{F}_{0}(z) \stackrel{(*)_{2}}{=}  F_{0}(z). 
\end{split}
\end{align}
Here, $(*)_{2}$ is a direct consequence of $z_{2}>0$ and \eqref{eq:extended}.  
To see $(*)_{1}$ note that $\{p\}=\{\nabla F_{0}(z)\}\subset K\subset\mathrm{dom}(\widetilde{F}_{0}^{*})$ by \ref{item:coinc1}. Moreover, $z\in\mathrm{dom}(\widetilde{F}_{0})$, and so the Fenchel--Legendre identity \eqref{eq:FenchelLegendre} gives 
\begin{align*}
z\cdot p =\widetilde{F}_{0}(z)+\widetilde{F}_{0}^{*}(p), 
\end{align*}
which is $(*)_{1}$; in particular, the involved quantities are finite. 

For the other direction, we claim that $F_{0}^{*}=\widetilde{F}_{0}^{*}\leq G_{0}$.  Assuming this for the time being, the proof is concluded by the order-reversing property of Fenchel conjugates. More precisely, we estimate for $q\in \overline{K}$ by use of Fenchel's inequality: 
\begin{align*}
z\cdot q -G_{0}(q) \leq z\cdot q - F_{0}^{*}(q) = z\cdot q - \widetilde{F}_{0}^{*}(q)\leq \widetilde{F}_{0}(z) = F_{0}(z). 
\end{align*}
Passing to the supremum over all $q\in\overline{K}$, we find 
\begin{align}\label{eq:moon2}
    F(z) \stackrel{\eqref{eq:Ffinal}}{=}  G_{0}^{*}(z) \leq F_{0}(z). 
\end{align}
Hence, \eqref{eq:moon1} and \eqref{eq:moon2} combine to $F(z)=F_{0}(z)$ for $z\in\mathcal{C}_{R,\delta}$. It thus remains to prove the inequality $\widetilde{F}_{0}^{*}\leq G_{0}$. By Remark \ref{rem:extended}, this is equivalent to $F_{0}^{*}\leq G_{0}$. 

On $\partial K$, we have $F_{0}^{*}(p)\leq 0 =G_{0}(p)$, and $F_{0}^{*}(p)\leq \infty =G_{0}(p)$ trivially holds on $\R^{2}\setminus\overline{K}$. Lastly, on $K$, we conclude similarly as above that  
\begin{align*}
h(L-p_{2})\stackrel{\text{\ref{item:(i)}}}{\leq}L-p_{2} & \Longrightarrow L-\frac{p_{1}^{2}}{4b}-p_{2} \geq h(L-p_{2})-\frac{p_{1}^{2}}{4b} \\ 
& \!\!\!\stackrel{0<\alpha<1}{\Longrightarrow} \Big(L-\frac{p_{1}^{2}}{4b}-p_{2} \Big)^{\alpha} \geq \rho(p)^{\alpha} \\ 
& \Longrightarrow F_{0}^{*}(p)\stackrel{\eqref{eq:wildestdreams},\,\eqref{eq:inc}}{=} -C(A,\mu)\Big(L-\frac{p_{1}^{2}}{4b}-p_{2} \Big)^{\alpha}  \\ & \;\;\;\;\;\;\;\;\;\;\;\;\;\;\;\;\;\;\;\;\;\;\;\;\leq -C(A,\mu)\rho(p)^{\alpha} \stackrel{\eqref{eq:swallowtears}}{=}G_{0}(p).
\end{align*}
This establishes $F_{0}^{*}\leq G_{0}$ globally on $\R^{2}$, and the proof of \ref{item:coinc2} is complete. 
\end{proof}
We are now ready to summarize the preceding results to give the: 
\begin{proof}[Proof of Theorem \ref{thm:muellipticextension}]
By Proposition \ref{prop:Fmuell}, $F\in\hold^{2}(\R^{2})$ is $\mu$-elliptic and of linear growth. Now, Proposition \ref{prop:coincidence} asserts that there exists $\delta_{0}>0$ such that, for any $0<\delta<\delta_{0}$, there exists $R=R(\delta)>0$ with $F=F_{0}$ on $\mathcal{C}_{R,\delta}$. The proof is complete. 
\end{proof}
 \section{Construction of a singular minimizer}\label{sec:singularity}
 We now proceed to the proof of Theorem \ref{thm:main}. Throughout, let $\mu>3$ and define 
 \begin{align}\label{eq:randkappa}
r\coloneqq \frac{\mu-3}{\mu-1}\in (0,1)\;\;\;\text{and}\;\;\;\kappa \coloneqq \frac{A(\mu-2)}{b}r^{1-\mu}.
 \end{align}
 Moreover, we record in advance that  
 \begin{align}\label{eq:justifier}
    r-1 = -\frac{2}{\mu-1}\;\;\;\text{and so}\;\;\;(r-1)(1-\mu)=2.
 \end{align}
 \subsection{The absolutely continuous part}
 Our overall plan is to construct a $\bv$-function whose approximate gradient is contained in a suitable cone  $\mathcal{C}_{R,\delta}$ where $F=F_{0}$. To this end, we begin with: 
 \begin{lemma}\label{lem:wconstruct}
With $r$ and $\kappa$ as in \eqref{eq:randkappa}, there exists $\ell>0$ such that the ordinary differential equation 
\begin{align}\label{eq:ODE}
\begin{cases}
w'' + \kappa w^{\mu-2} = 0&\;\text{in}\;(-\ell,\ell),\\ 
w(0)=1,\;\;\;w'(0)=0
\end{cases}
\end{align}
has a unique solution $w\in\hold^{2}((-\ell,\ell);\R_{>0})$. 
 \end{lemma}
\begin{proof}
We write \eqref{eq:ODE} as $Y'=\mathcal{F}(Y)$ and $Y(0)=(1,0)$, where 
\begin{align*}
Y=(w,v)\;\;\;\text{and}\;\;\;\mathcal{F}(Y)=\mathcal{F}(w,v)=(v,-\kappa w^{\mu-2}). 
\end{align*}
Clearly, $\mathcal{F}$ is Lipschitz in an open neighbourhood of $Y(0)$. Hence, the Picard-Lindel\"{o}f theorem gives the conclusion of the lemma for some $\ell>0$; diminishing $\ell$ if necessary, we may assume that $w>0$ on $(-\ell,\ell)$ due to $w(0)=1$. The proof is complete. 
\end{proof}
In the following, we fix $\ell>0$ and $w\in\hold^{1}((-\ell,\ell);\R_{>0})$ as provided by Lemma \ref{lem:wconstruct}. We define 
\begin{align}\label{eq:dexter}
a(x)\coloneqq w(x)^{-r},\qquad x\in (-\ell,\ell). 
\end{align}
Towards the construction of the requisite singular minimizer, we define 
\begin{align}\label{eq:dexter1}
v(x,y)=a(x)\mathrm{sgn}(y)|y|^{r},\qquad x\in (-\ell,\ell),\;y\neq 0.
\end{align}
For future reference, we compute for $x\in (-\ell,\ell)$ and $y\neq 0$: 
\begin{align}\label{eq:derivissimo1}
\partial_{x}v(x,y) = a'(x)\mathrm{sgn}(y)|y|^{r}\;\;\;\text{and}\;\;\;\partial_{y}v(x,y)=r a(x)|y|^{r-1}.
\end{align}
Recall that, by Lemma \ref{lem:wconstruct}, $w>0$ on $(-\ell,\ell)$, whereby $a>0$ on $(-\ell,\ell)$ too. It is then convenient to observe that 
\begin{align}\label{eq:observer}
\frac{\partial_{x}v(x,y)}{\partial_{y}v(x,y)} & = \frac{a'(x)}{ra(x)}y\;\;\;\text{and}\;\;\;\frac{a'(x)}{ra(x)}=-\frac{w'(x)}{w(x)}\qquad\text{if}\;y\neq 0. 
\end{align}
To construct our underlying domain $\Omega\subset\R^{2}$, we note that (after diminishing $\ell>0$ again if necessary) there exist constants $m,M,M'>0$ such that 
\begin{align}\label{eq:abound}
m\leq a(x)\leq M\;\;\;\text{and}\;\;\;|a'(x)|\leq M'\qquad\text{for all}\;x\in (-\ell,\ell). 
\end{align}
This is a direct consequence of the $\hold^{1}$-regularity of $a$ and $w>0$ on the interval provided by Lemma \ref{lem:wconstruct}. We now have:
\begin{lemma}\label{lem:contained1}
Let $R\geq 1$ and $\delta>0$. Then there exists $\varepsilon>0$ such that 
\begin{align*}
\nabla v(x,y)\in\mathcal{C}_{R,\delta}\qquad\text{for all}\;x\in (-\ell,\ell)\;\text{and}\;y\in (-\varepsilon,\varepsilon)\setminus\{0\}.
\end{align*}
\end{lemma}
\begin{proof}
Put 
\begin{align}
\varepsilon \coloneqq \min\Big\{\Big(\frac{rm}{R}\Big)^{\frac{1}{1-r}},\frac{\delta rm}{M'}\Big\}. 
\end{align}
Now, if $y\in (-\varepsilon,\varepsilon)\setminus\{0\}$, then we have for all $x\in (-\ell,\ell)$:
\begin{align}\label{eq:cannon}
\begin{split}
|y|< \Big(\frac{rm}{R}\Big)^{\frac{1}{1-r}} & \stackrel{0<r<1}{\Longrightarrow}  |y|^{1-r}<\frac{rm}{R} \\ &\;\;\Longrightarrow R<rm|y|^{r-1} \stackrel{\eqref{eq:abound}}{\leq} ra(x)|y|^{r-1} \stackrel{\eqref{eq:derivissimo1}}{=} \partial_{y}v(x,y).
\end{split}
\end{align}
On the other hand, by \eqref{eq:observer}, it follows that 
\begin{align}\label{eq:cannonball}
\left\vert\frac{\partial_{x}v(x,y)}{\partial_{y}v(x,y)}\right\vert \leq \frac{M'}{rm}|y| <\frac{M'}{rm}\frac{\delta rm}{M'}=\delta.
\end{align}
Since $\partial_{y}v(x,y)=ra(x)|y|^{r-1}>0$, \eqref{eq:cannon} and \eqref{eq:cannonball} combine to $\nabla v(x,y)\in\mathcal{C}_{R,\delta}$ for $x\in (-\ell,\ell)$ and $y\in(-\varepsilon,\varepsilon)\setminus\{0\}$, as required. The proof is complete. 
\end{proof}
\subsection{Construction of the singular $\bv$-minimizer} 
Since $\mu>3$, we may pick the number $\delta_{0}>0$ from Theorem \ref{thm:muellipticextension} and fix $0<\delta<\delta_{0}$. In consequence, Theorem \ref{thm:muellipticextension} yields $R=R(\delta)>1$ such that there exists a $\mu$-elliptic $\hold^{2}$-integrand $F\in\hold^{2}(\R^{2})$ with $F=F_{0}$ on $\mathcal{C}_{R,\delta}$. Moreover, by Proposition \ref{prop:coincidence}\ref{item:coinc1}, we may assume that 
\begin{align}\label{eq:container}
\nabla F_{0}(z)=\nabla F(z)\in K\;\;\;\qquad\text{for all}\;z\in\mathcal{C}_{R,\delta}. 
\end{align}
With these choices of $\delta$ and $R$, we let $\ell>0$ be as in \eqref{eq:abound} and let $\varepsilon>0$ be as in Lemma \ref{lem:contained1}. 

For $-\ell<x<\ell$ and $-\varepsilon<y<\varepsilon$ with $y\neq 0$, we then conclude from \eqref{eq:derivissimo1} that 
\begin{align}\label{eq:derivissimo2}
|\nabla v(x,y)|=(|a'(x)|^{2}|y|^{2r} + r^{2}a(x)^{2}|y|^{2r-2})^{\frac{1}{2}}=|y|^{r-1}(|a'(x)|^{2}|y|^{2}+|a(x)|^{2}r^{2})^{\frac{1}{2}}
\end{align}
and therefore 
\begin{align}\label{eq:derivissimo3}
|y|^{r-1}mr\leq |\nabla v(x,y)|\leq |y|^{r-1}\sqrt{(M')^{2}\varepsilon^{2}+M^{2}r^{2}}. 
\end{align}
Now, \eqref{eq:derivissimo3} implies that 
\begin{align}\label{eq:mainintegrability}
\nabla v\in\lebe^{1}((-\ell,\ell)\times(-\varepsilon,\varepsilon);\R^{2})\Longleftrightarrow \int_{0}^{\varepsilon}y^{r-1}\,\dif y <\infty \Longleftrightarrow r>0 \stackrel{\eqref{eq:randkappa}}{\Longleftrightarrow} \mu>3. 
\end{align}
We subsequently put 
\begin{align}\label{eq:Omegadef}
\Omega\coloneqq \Big(-\frac{\ell}{2},\frac{\ell}{2}\Big) \times \Big(-\frac{\varepsilon}{2},\frac{\varepsilon}{2}\Big)\;\;\;\text{and}\;\;\; U\coloneqq (-\ell,\ell)\times(-\varepsilon,\varepsilon).
\end{align}
Next note that $v$ is continuous on $U$, whereby $\mu>3$ implies that $
v\in\sobo^{1,1}(U)$. We fix a jump height $J>0$ and we eventually define the requisite map by
\begin{align}\label{eq:ufinal}
\text{\fbox{$u(x,y)\coloneqq v(x,y)+J\mathbbm{1}_{U\cap\{y>0\}}(x,y)$,}}
\end{align}
which we initially consider on $U$. Since $v\in\sobo^{1,1}(U)$, $u\in\bv(U)\setminus\sobo^{1,1}(U)$. Moreover, crucially, $v$ itself solves the Euler--Lagrange equation, and the associated stress is continuous: 
\begin{proposition}\label{prop:stressedout}
Subject to the choices of $\delta,R,\varepsilon$ and $\ell$ fixed at the beginning of the present subsection, the associated stress 
\begin{align}\label{eq:stresscontinuity}
\sigma\coloneqq \nabla F(\nabla v)\qquad\text{\emph{is continuous on $U$}}
\end{align}
and is solenoidal in the sense of distributions: 
\begin{align}\label{eq:stresssolenoidality}
\mathrm{div}(\sigma)=0\qquad\text{in}\;\mathscr{D}'(U).
\end{align}
\end{proposition}
\begin{proof}
Since $\ell,\varepsilon>0$ have been adjusted due to \eqref{eq:abound} and Lemma \ref{lem:contained1}, we have $\nabla v(x,y)\in\mathcal{C}_{R,\delta}$ for every $x\in (-\ell,\ell)$ and $y\in (-\varepsilon,\varepsilon)\setminus\{0\}$. For such values of $(x,y)$, we thus have $F(\nabla v(x,y))=F_{0}(\nabla v(x,y))$ and so $\sigma(x,y)=\nabla  F_{0}(\nabla v(x,y))$. 

Based on this observation, the continuity assertion on $\sigma$ is a consequence of \eqref{eq:firstvariation}:
\begin{align}
\sigma(x,y) & \stackrel{\eqref{eq:firstvariation}}{=} \left(\begin{array}{c} \frac{2b\partial_{x}v(x,y)}{\partial_{y}v(x,y)} \\ L + (2-\mu)A(\partial_{y}v(x,y))^{1-\mu}-\frac{b \partial_{x}v(x,y)^{2}}{\partial_{y}v(x,y)^{2}}\end{array} \right) \\ 
& \!\!\!\!\!\!\!\stackrel{\eqref{eq:justifier},\,\eqref{eq:observer}_{1}}{=} 
\left(\begin{array}{c} 2b\frac{a'(x)}{ra(x)}y \\ L + (2-\mu)A(ra(x))^{1-\mu}|y|^{2}-b \Big(\frac{a'(x)}{ra(x)}\Big)^{2}y^{2}\end{array} \right) \label{eq:sigmacompute}\\ 
& \!\!\! \stackrel{\eqref{eq:observer}_{2}}{=}  \left(\begin{array}{c} -2b\frac{w'(x)}{w(x)}y \\ L + (2-\mu)A(rw(x)^{-r})^{1-\mu}|y|^{2}-b \Big(\frac{w'(x)}{w(x)}\Big)^{2}y^{2}\end{array} \right).
\end{align}
Clearly, $\sigma$ is continuous up to the boundary away from $\{y=0\}$. Now, if $|y|\searrow 0$, then \eqref{eq:sigmacompute} shows that $\sigma(x,y)\to(0,L)$ uniformly in $x$. From here, the claimed continuity assertion \eqref{eq:stresscontinuity} follows. 

Towards \eqref{eq:stresssolenoidality}, let $y\neq 0$. We then have 
\begin{align}
\mathrm{div}(\sigma(x,y)) & = -2b\Big(\frac{w''(x)}{w(x)}-\Big(\frac{w'(x)}{w(x)}\Big)^{2}\Big)y  + 2(2-\mu)A(rw(x)^{-r})^{1-\mu}y-2b \Big(\frac{w'(x)}{w(x)}\Big)^{2}y \notag\\ 
& = -2by\frac{w''(x)}{w(x)} + 2(2-\mu)A(rw(x)^{-r})^{1-\mu}y \label{eq:rita}\\ 
& \!\!\!\stackrel{\eqref{eq:ODE}}{=} 2\kappa by w(x)^{\mu-3} + 2(2-\mu)A(rw(x)^{-r})^{1-\mu}y \eqqcolon \mathrm{IV} = 0.  \notag
\end{align}
To see the last identity, note that $-r(1-\mu)=\mu-3$ by $\eqref{eq:randkappa}_{1}$, whereby $\eqref{eq:randkappa}_{2}$ yields  
\begin{align*}
\mathrm{IV} & = 2y \Big(\kappa b w(x)^{\mu-3} + (2-\mu)A(rw(x)^{-r})^{1-\mu} \Big) \\ & \!\!\!\stackrel{\eqref{eq:randkappa}_{2}}{=} 2y \Big(\frac{A(\mu-2)}{b}r^{1-\mu} b + (2-\mu)Ar^{1-\mu} \Big)w(x)^{\mu-3} = 0. 
\end{align*}
In conclusion, $\mathrm{div}(\sigma(x,y))=0$ whenever $y\neq 0$. To arrive at \eqref{eq:stresssolenoidality}, let $\varphi\in\hold_{c}^{\infty}(U)$ be arbitrary. Writing $U^{\pm}\coloneqq \{(x,y)\in U\colon\;\pm y>0\}$, an integration by parts gives us 
\begin{align*}
\int_{U}\sigma\cdot\nabla\varphi\,\,\dif(x,y)& = \int_{U^{+}}\sigma\cdot\nabla\varphi\,\,\dif(x,y) + \int_{U^{-}}\sigma\cdot\nabla\varphi\,\,\dif(x,y) \\ 
& \!\!\!\!\stackrel{\eqref{eq:rita}}{=} \int_{\partial U^{+}\cap\{y=0\}}\varphi\sigma\cdot\nu_{\partial U^{+}}\,\,\dif\mathscr{H}^{1}(x,y) + \int_{\partial U^{-}\cap\{y=0\}}\varphi\sigma\cdot\nu_{\partial U^{-}}\,\,\dif\mathscr{H}^{1}(x,y) \\ & = 0.
\end{align*}
Here, the last equality holds because $\sigma(x,0)=(0,L)$ and the outer unit normals satisfy $\nu_{\partial U^{+}}=-\nu_{\partial U^{-}}$ along $\mathrm{spt}(\varphi)\cap\{y=0\}$. By arbitrariness of $\varphi$, we arrive at \eqref{eq:stresssolenoidality} and so the proof is complete. 
\end{proof}
We now come to the proof of our Main Theorem \ref{thm:main} in the two dimensional case. For $N=1$ and $n=2$, it is indeed a direct consequence of the following result:
\begin{theorem}[$\bv$-minimality of $u$]\label{thm:main2a}
Let $\mu>3$, and let $\delta,R,\varepsilon$ and $\ell$ be as fixed as in the beginning of the present subsection. Moreover, define $\Omega$ as in \eqref{eq:Omegadef} and let $F\in\hold^{2}(\R^{2})$ be the $\mu$-elliptic integrand as specified in the beginning of the subsection. Then $u\colon\Omega\to\R$ from \eqref{eq:ufinal} belongs to $(\bv(\Omega)\cap\lebe^{\infty}(\Omega))\setminus\sobo^{1,1}(\Omega)$ and is a \emph{$\mathrm{BV}$-minimizer} of the functional 
\begin{align*}
\mathscr{F}[\psi;\Omega]=\int_{\Omega}F(\nabla \psi)\,\,\dif x
\end{align*}
\emph{with respect to its own boundary values $\mathrm{tr}_{\partial\Omega}(u)$}. In particular, there exists $u_{0}\in\sobo^{1,1}(\Omega)$ such that 
\begin{align}\label{eq:brocolli}
\overline{\mathscr{F}}_{u_{0}}^{*}[u;\Omega] \leq \overline{\mathscr{F}}_{u_{0}}^{*}[\psi;\Omega]\qquad\text{for all}\;\psi\in\bv(\Omega). 
\end{align}
\end{theorem}
\begin{proof}
We start by noting that $u$ is a priori also defined on the set $U$ from \eqref{eq:Omegadef}, on which $\sigma=\nabla F(\nabla v)$ is continuous and distributionally solenoidal by Proposition \ref{prop:stressedout}.

Let $\psi\in\bv(\Omega)$ be arbitrary with Lebesgue--Radon--Nikod\'{y}m decomposition $\D\psi=\nabla\psi \mathscr{L}^{2}+\D^{s}\psi$ of its measure gradient. By convexity of $F$, Fenchel's inequality gives us  
\begin{align}\label{eq:combo1}
F(\nabla \psi)\geq \sigma\cdot\nabla \psi - F^{*}(\sigma)\qquad\text{$\mathscr{L}^{2}$-a.e. in $\Omega$}.
\end{align}
On the other hand, $G_{0}\colon\R^{2}\to\R\cup\{+\infty\}$ from \eqref{eq:swallowtears} is convex, proper and lower semicontinuous. Hence, by the Fenchel--Moreau Theorem (see Lemma \ref{lem:fenchelmoreau}), $F=G_{0}^{*}$ gives us $F^{*}=G_{0}^{**}=G_{0}$. In particular, by \eqref{eq:swallowtears}, $\mathrm{dom}(F^{*})=\overline{K}$. Now, Lemma \ref{lem:contained1} gives us $\nabla v(x,y)\in\mathcal{C}_{R,\delta}$ for $(x,y)\in U$ with $y\neq 0$, and so $\sigma=\nabla F_{0}(\nabla v)=\nabla F(\nabla v)\in K\subset\mathrm{dom}(F^{*})$ on $\Omega\setminus\{y=0\}$ by \eqref{eq:container}. Thus, Lemma \ref{lem:fenchel} gives us 
\begin{align*}
F^{\infty}(\xi)= \sup_{\eta\in{\mathrm{dom}(F^{*})}}\eta\cdot\xi \geq \sigma(x,y)\cdot\xi\qquad\text{for all $(x,y)\in\overline{\Omega}\setminus\{y=0\}$ and all $\xi\in\R^{2}$}. 
\end{align*}
For $x\in (-
\frac{\ell}{2},\frac{\ell}{2})$, the continuity of $\sigma$ allows us to send $|y|\searrow 0$ in the preceding inequality. Thus, 
\begin{align}\label{eq:doakes}
F^{\infty}(\xi)\geq  \sigma(x,y)\cdot\xi\qquad\text{for all $(x,y)\in\overline{\Omega}$ and all $\xi\in\R^{2}$}, 
\end{align}
where, for $(x,y)=(x,0)$, \eqref{eq:doakes} reads 
\begin{align}\label{eq:doakesAA}
F^{\infty}(\xi)\geq  (0,L)\cdot\xi\qquad\text{for all  $\xi\in\R^{2}$}. 
\end{align}
As a first consequence, we thus obtain 
\begin{align}\label{eq:combo2}
\int_{\Omega}F^{\infty}\Big(\,\frac{\dif \mathrm{D}^{s}\psi}{\dif|\mathrm{D}^{s}\psi|}\Big)\,\dif|\mathrm{D}^{s}\psi| \stackrel{\eqref{eq:doakes}}{\geq} \int_{\Omega}\sigma\cdot\,\dif\mathrm{D}^{s}\psi.
\end{align}
Now, \eqref{eq:combo1} and \eqref{eq:combo2} combine to the bulk inequality
\begin{align}
\int_{\Omega}F(\nabla \psi)\,\,\dif (x,y) + \int_{\Omega}F^{\infty}\Big(\,\frac{\dif \mathrm{D}^{s}\psi}{\dif|\mathrm{D}^{s}\psi|}\Big)\,\,\dif|\mathrm{D}^{s}\psi| & \geq \int_{\Omega}\sigma\cdot\nabla \psi-F^{*}(\sigma)\,\,\dif (x,y) + \int_{\Omega}\sigma\cdot\,\dif\mathrm{D}^{s}\psi \notag\\ 
& =  \int_{\Omega}\sigma\cdot\,\dif\mathrm{D}\psi - \int_{\Omega}F^{*}(\sigma)\,\,\dif (x,y).\label{eq:boyd}
\end{align}
Now let $u_{0}\in\sobo^{1,1}(\Omega)$ be such that $\mathrm{tr}_{\partial\Omega}(u_{0})=\mathrm{tr}_{\partial\Omega}(u)$ $\mathscr{H}^{1}$-a.e. on $\partial\Omega$. Even though $u\in\bv(\Omega)\setminus\sobo^{1,1}(\Omega)$, this is possible since the trace operator $\mathrm{tr}_{\partial\Omega}\colon\sobo^{1,1}(\Omega)\to\lebe^{1}(\partial\Omega)$ is surjective. By \eqref{eq:doakes}, this implies that 
\begin{align}\label{eq:lumen}
F^{\infty}((\mathrm{tr}_{\partial\Omega}(u_{0})-\mathrm{tr}_{\partial\Omega}(\psi))\nu_{\partial\Omega}) & \geq (\mathrm{tr}_{\partial\Omega}(u_{0})-\mathrm{tr}_{\partial\Omega}(\psi))\sigma\cdot\nu_{\partial\Omega}\qquad\mathscr{H}^{1}\text{-a.e. on $\partial\Omega$}. 
\end{align}
By our definition of $\Omega$ and $U$ from \eqref{eq:Omegadef}, we may extend $\psi$ to a (non-relabelled) function $\psi\in\bv(U)$. Since $\sigma\colon U\to \R^{2}$ is continuous and distributionally solenoidal, Lemma \ref{lem:IBP} yields  
\begin{align}\label{eq:laguerta}
\int_{\Omega}\sigma\cdot\,\,\dif\D\psi = \int_{\partial\Omega}\mathrm{tr}_{\partial\Omega}(\psi)\sigma\cdot\nu_{\partial\Omega}\,\,\dif\mathscr{H}^{1}. 
\end{align}
In particular, we arrive at the following bound: 
\begin{align}\label{eq:wcompare}
\begin{split}
\overline{\mathscr{F}}_{u_{0}}^{*}[\psi;\Omega] & \stackrel{\eqref{eq:boyd},\,\eqref{eq:lumen}}{\geq} \int_{\Omega}\sigma\cdot\,\dif\mathrm{D}\psi - \int_{\Omega}F^{*}(\sigma)\,\,\dif (x,y) + \int_{\partial\Omega} (\mathrm{tr}_{\partial\Omega}(u_{0})-\mathrm{tr}_{\partial\Omega}(\psi))\sigma\cdot\nu_{\partial\Omega}\,\,\dif\mathscr{H}^{1} \\ 
& \;\;\;\;\stackrel{\eqref{eq:laguerta}}{=} - \int_{\Omega}F^{*}(\sigma)\,\,\dif (x,y) + \int_{\partial\Omega}\mathrm{tr}_{\partial\Omega}(u_{0})\sigma\cdot\nu_{\partial\Omega}\,\,\dif\mathscr{H}^{1}. 
\end{split}
\end{align}
With $u$ as in \eqref{eq:ufinal}, we let  $\D u=\nabla u\mathscr{L}^{2} + \D^{s}u$ be its Lebesgue--Radon-Nikod\'{y}m decomposition of its measure gradient. The very definition \eqref{eq:ufinal} particularly implies that $\nabla u=\nabla v$ $\mathscr{L}^{2}$-a.e. in $\Omega$ (in fact, everywhere away from $\{y=0\}$). Therefore, we have 
\begin{align*}
F(\nabla u)=F(\nabla v)=\sigma\cdot\nabla v - F^{*}(\sigma)=\sigma\cdot\nabla u - F^{*}(\sigma)\qquad\mathscr{L}^{2}\text{-a.e. in $\Omega$}
\end{align*}
by the case of equality in Fenchel's inequality. Hence, 
\begin{align}\label{eq:fenchel1A}
\int_{\Omega}F(\nabla u)\,\,\dif (x,y) = \int_{\Omega}\sigma\cdot\nabla u\,\,\dif (x,y) - \int_{\Omega}F^{*}(\sigma)\,\,\dif (x,y).
\end{align}
As a crucial point, we now note that 
\begin{align}\label{eq:densitycontained}
\D^{s}u = Je_{2}\mathscr{H}^{1}\mres (\Omega\cap\{y=0\}).
\end{align}
By the very construction of the cones $\mathcal{C}_{R,\delta}$, see Figure \ref{fig:cone}, $\frac{1}{t}e_{2}\in\mathcal{C}_{R,\delta}$ holds  for all sufficiently small $t>0$. Since $F=F_{0}$ on $\mathcal{C}_{R,\delta}$, we thus obtain 
\begin{align}\label{eq:baptista}
F^{\infty}(e_{2})=\lim_{t\searrow 0}tF\Big(\frac{e_{2}}{t}\Big)= \lim_{t\searrow 0}tF_{0}\Big(\frac{e_{2}}{t}\Big) \stackrel{\eqref{eq:main}}{=} L.
\end{align}
By \eqref{eq:densitycontained},  \eqref{eq:baptista} and the positive $1$-homogeneity of $F^{\infty}$, see \eqref{eq:1homorecess}, it follows that 
\begin{align}\label{eq:recessionidentity}
\begin{split}
\int_{\Omega}F^{\infty}\Big(\,\frac{\dif\D^{s}u}{\dif|\D^{s}u|}\Big)\,\,\dif|\D^{s}u| & \stackrel{\eqref{eq:densitycontained}}{=} J\int_{\Omega\cap\{y=0\}}F^{\infty}(e_{2})\,\,\dif\mathscr{H}^{1} \\ & \stackrel{\eqref{eq:baptista}}{=} JL\mathscr{H}^{1}(\Omega\cap\{y=0\}) = JL\ell.
\end{split}
\end{align}
Recalling that $\sigma$ is continuous with $\sigma=(0,L)$ on $\Omega\cap\{y=0\}$, we have 
\begin{align}\label{eq:fenchelsingular}
\begin{split}
\int_{\Omega}\sigma\cdot\,\,\dif\D^{s}u & \stackrel{\eqref{eq:densitycontained}}{=} J\int_{\Omega\cap\{y=0\}}\sigma\cdot e_{2}\,\,\dif\mathscr{H}^{1}  = JL\ell\\ & \stackrel{\eqref{eq:recessionidentity}}{=} \int_{\Omega}F^{\infty}\Big(\,\frac{\dif\D^{s}u}{\dif|\D^{s}u|}\Big)\,\,\dif|\D^{s}u|
\end{split}
\end{align}
To conclude the proof, we note that \eqref{eq:laguerta} also holds with $\psi$ being replaced by $u$. In consequence, we find that
\begin{align*}
\int_{\Omega}F(\nabla u)\,\,\dif (x,y) + \int_{\Omega}F^{\infty}\Big(\,\frac{\dif\D^{s}u}{\dif|\D^{s}u|} \Big)\,\,\dif|\D^{s}u| & \stackrel{\eqref{eq:fenchel1A},\,\eqref{eq:fenchelsingular}}{=} \int_{\Omega}\sigma\cdot\,\dif\D u-\int_{\Omega}F^{*}(\sigma)\,\,\dif(x,y) \\ 
& \;\;\;\;\stackrel{\eqref{eq:laguerta}}{=} \int_{\partial\Omega}\mathrm{tr}_{\partial\Omega}(u_{0})\sigma\cdot\nu_{\partial\Omega}\,\,\dif\mathscr{H}^{1} \\ & \;\;\;\;\;\;\,- \int_{\Omega}F^{*}(\sigma)\,\,\dif (x,y),  
\end{align*}
where we used that $\mathrm{tr}_{\partial\Omega}(u_{0})=\mathrm{tr}_{\partial\Omega}(u)$ $\mathscr{H}^{1}$-a.e. on $\partial\Omega$. The latter also gives us 
\begin{align}\label{eq:quinn}
\overline{\mathscr{F}}_{u_{0}}^{*}[u;\Omega] = \int_{\partial\Omega}\mathrm{tr}_{\partial\Omega}(u_{0})\sigma\cdot\nu_{\partial\Omega}\,\,\dif\mathscr{H}^{1} - \int_{\Omega}F^{*}(\sigma)\,\,\dif (x,y),
\end{align}
Going back to \eqref{eq:wcompare} and recalling that $\psi\in\bv(\Omega)$ was arbitrary, \eqref{eq:quinn} gives us 
\begin{align*}
\overline{\mathscr{F}}_{u_{0}}^{*}[\psi;\Omega] \geq \overline{\mathscr{F}}_{u_{0}}^{*}[u;\Omega]\qquad\text{for all}\;\psi\in\bv(\Omega). 
\end{align*}
This is \eqref{eq:brocolli}, and since $u\in(\lebe^{\infty}(\Omega)\cap\bv(\Omega))\setminus\sobo^{1,1}(\Omega)$, the proof is complete. 
\end{proof}
We conclude this section with a remark. 
\begin{remark}
The reader might notice that the reduction to minimality of $u$ for $\overline{\mathscr{F}}_{u_{0}}^{*}[-;\Omega]$ with respect to its own boundary values strongly relies on the fact that the prescribed traces are discontinuous. Therefore, in particular, Theorem \ref{thm:main} does not make any assertion on the non-$\sobo^{1,1}$-regularity of $\bv$-minimizers (even for $\mu>3$) if the boundary data belong to a better space than the trace space $\lebe^{1}$ of $\bv$ or $\sobo^{1,1}$, respectively. 
\end{remark}

\medskip 
\noindent \textbf{AI-Statement.} The mathematical arguments displayed in the present paper are due to the author. Chat GPT (GPT-6 Astra) has been used for some of the BibTex files, and to find typos and grammar errors in a preliminary version of the manuscript. Moreover, Google Gemini (Gemini 3.8 Flash) helped with fitting the convex set $K$ to the parabola in TikZ, see Figure \ref{fig:mod}; apart from that, the figures contained in the present manuscript have been created by the author exclusively. Lastly, Chat GPT and Gemini (in the versions displayed above) have been used to check the mathematical arguments on September 23, 2026. 
\bibliographystyle{alpha}		
\bibliography{references}

\end{document}